\documentclass[11pt,twoside]{amsart}
\usepackage{amssymb,amsbsy,amsmath,amsfonts,amscd}
\usepackage{graphicx,color}
\usepackage{float,url}
\usepackage[colorlinks,linkcolor=red,citecolor=blue]{hyperref}
\usepackage{booktabs}
\usepackage{bbm}
\usepackage{verbatim}
\usepackage{cancel} 
\usepackage{ulem}
\usepackage{mathrsfs} 
\usepackage{footnote}
\usepackage{subfigure}
\usepackage{enumerate}
\usepackage{enumitem}
\usepackage{nameref}
\usepackage{titletoc}

\allowdisplaybreaks[4]

\newcommand{\commentout}[1]{}

\newcommand {\vp} {\varphi}
\newcommand {\lb} {\lambda}

\newcommand {\dv}  { {\rm div} }
\newcommand {\cae} { {\mathcal E} }

\newcommand {\f}   {\frac}
\newcommand {\p}   {\partial}

\newcommand{\dis}{\displaystyle}
\newcommand{\beq}{\begin{equation}}
\newcommand{\eeq}{\end{equation}}
\newcommand{\bea} {\begin{array}{rl}}
\newcommand{\eea} {\end{array}}
\newcommand{\bepa}{\left\{ \begin{array}{l}}
\newcommand{\eepa} {\end{array}\right.}
\allowdisplaybreaks[4]
\newtheorem{theorem}{Theorem}
\newtheorem{lemma}[theorem]{Lemma}
\newtheorem{definition}[theorem]{Definition}
\newtheorem{remark}[theorem]{Remark}
\newtheorem{proposition}[theorem]{Proposition}

\newtheorem{notation}[theorem]{Notation}

\newcommand{\RNum}[1]{\uppercase\expandafter{\romannumeral #1\relax}}
\DeclareSymbolFont{bbold}{U}{bbold}{m}{n}
\DeclareSymbolFontAlphabet{\mathbbold}{bbold}
\title{Three asymptotic regimes for a Keller--Segel system\\
with volume-filling effect}
\author[M. Zhang]{Mingyue Zhang}
\thanks{Institute of Analysis and Scientific Computing, TU Wien, Wiedner Hauptstra\ss e 8--10, 1040 Wien, Austria}
\thanks{Email: mingyue.zhang@tuwien.ac.at} 

\begin{document}

\begin{abstract}
We study three asymptotic regimes for a parabolic--elliptic Keller--Segel system with porous-medium diffusion and the volume-filling sensitivity function $u(1-u)$. First, when $D=\epsilon^2\to0$, $\delta=1$, $1\le m<2$, and the chemotactic sensitivity coefficient is below an explicit $m$-dependent parabolicity threshold, we prove strong convergence to a scalar nonlinear diffusion equation. Second, as $\delta=\epsilon\to0$ with $D=1$, we obtain a hyperbolic--elliptic Keller--Segel limit by means of a kinetic reduction argument. Third, under the scaling $\chi=\epsilon^{-1}$, $D=\epsilon^2$, $\delta=1$, and for characteristic initial data of finite perimeter, we prove convergence on fixed time intervals to the time-independent initial patch in $BV(\Omega;\{0,1\})$. Because the sensitivity function is nonlinear in the density, weak convergence alone is insufficient to identify the chemotactic flux. Strong compactness of the density is therefore required in all three regimes and is obtained, respectively, through energy--entropy estimates, kinetic reduction, and $BV$ compactness.

\end{abstract}
\maketitle

\noindent{\makebox[1in]\hrulefill}\newline
\section*{Introduction}
\label{sec:intro}

Chemotaxis is the directed movement of cells in response to chemical
signals. It occurs in many biological processes involving bacteria,
slime molds, skin pigmentation, and leukocytes, among other examples
\cite{murray2007mathematical}. The Patlak--Keller--Segel system,
introduced in \cite{patlak1953random,KellerSegel1970initiation}, is one
of the main mathematical models for chemotaxis. Depending on the
parameters, its solutions may exhibit aggregation, blow-up, pattern
formation, or stabilization
\cite{jager1992explosions,winkler2022family}.

In this paper, we study a Keller--Segel system in which nonlinear cell
diffusion competes with chemotactic aggregation. 
The chemotactic flux contains the volume-filling sensitivity function $u(1-u)$.
This function makes the chemotactic flux vanish at both $u=0$ and $u=1$.
The system is
\beq
\left\{
\begin{aligned}
&\p_t u-\delta\Delta\f{u^m}{m}
+\dv\big[\chi u(1-u)\nabla v\big]=0,
&&x\in\Omega,\quad t>0,
\\
&-D\Delta v+v=u,
&&x\in\Omega,\quad t>0,
\\
&\left(
\delta\nabla\f{u^m}{m}
-\chi u(1-u)\nabla v
\right)\cdot\vec n=0,
\qquad
\nabla v\cdot\vec n=0,
&&x\in\partial\Omega,\quad t>0,
\\
&u(0,x)=u^0(x),\qquad 0\le u^0\le1,
&&x\in\Omega.
\end{aligned}
\right.
\label{eq:KSDC}
\eeq
Here, $u(t,x)$ denotes the cell density and $v(t,x)$ denotes the
chemical concentration. The parameter $\delta>0$ measures the strength
of cell diffusion, $\chi>0$ is the chemotactic sensitivity coefficient, and $D>0$
is the chemical diffusion coefficient. The case $m=1$ corresponds to
linear diffusion, whereas $m>1$ corresponds to porous-medium diffusion.
The relative sizes of $\delta$, $\chi$, and $D$ determine the balance
between cell diffusion, chemotactic aggregation, and chemical
diffusion. Our aim is to understand how different balances between
these mechanisms lead to different effective dynamics.

The domain $\Omega\subset\mathbb R^n$ is bounded and has a smooth
boundary. We denote by $\vec n$ the unit outward normal vector and
assume that $|\Omega|=1$.

Integrating system \eqref{eq:KSDC} over $\Omega$, and using the Neumann
boundary conditions, we obtain
\beq
M:=\int_\Omega u(t,x)\,dx
=\int_\Omega v(t,x)\,dx
=\int_\Omega u^0(x)\,dx,
\qquad 0<M<1.
\label{es:Mass}
\eeq
We only consider solutions satisfying
\beq
0\le u(t,x),\;v(t,x)\le1.
\label{MaxPrinciple}
\eeq

Nonlinear Keller--Segel systems and volume-filling models have been
studied in many works; see, for example,
\cite{Maini2008,Hillen2009,Painter2018}. The surveys
\cite{Hillen2009,Painter2018} describe several biological applications
and mathematical approaches. For the system considered here, long-time
behavior and pattern formation for small mass were studied in
\cite{bpmz2024}. Global existence of weak solutions and an
incompressible limit under a general density constraint were obtained
in \cite{he2024in}. Several singular limits for related Keller--Segel
systems have also been considered in
\cite{he2024in,noemi2023,chen2020vanishing}.

The purpose of this paper is to study three singular regimes of
\eqref{eq:KSDC}. These regimes lead to three different limiting dynamics:
a nonlinear diffusion equation, a hyperbolic--elliptic system, and a
static sharp-interface limit. Although the three limits are different,
they share the same main difficulty. The nonlinear flux
\[
u_\epsilon(1-u_\epsilon)\nabla v_\epsilon
\]
cannot, in general, be identified from weak convergence alone.
Therefore, strong compactness of the density is needed in all three
regimes. We obtain this compactness by three different methods:
energy--entropy estimates, kinetic reduction, and $BV$ compactness.

\noindent$\bullet$ \textbf{Vanishing chemical diffusion and a nonlinear
diffusion limit.}
We first fix $\delta=1$ and consider
\[
D=\epsilon^2\to0,
\qquad
1\le m<2,
\qquad
0<\chi<
\frac{(3-m)^{3-m}}{(2-m)^{2-m}}.
\]
Formally, the elliptic equation gives $v_\epsilon\to u$, and the limit
equation is
\[
\partial_tu-\Delta\frac{u^m}{m}
+\dv\big[\chi u(1-u)\nabla u\big]=0.
\]
We prove that both $u_\epsilon$ and $v_\epsilon$ converge strongly to
the same limit $u$. The strong compactness follows from the energy and
entropy inequalities.

The explicit threshold $(3-m)^{3-m}/(2-m)^{2-m}$ also has a structural
meaning. The same threshold appears in the parabolicity condition and
in the uniqueness condition for the homogeneous steady state. Under the
strict inequality, the homogeneous steady state is also exponentially
stable.

For the classical parabolic--parabolic Keller--Segel system, a
vanishing chemical-diffusion limit was studied in
\cite{chen2020vanishing} by semigroup methods. A related singular limit
for a chemotaxis system with logarithmic sensitivity was considered in
\cite{wang2016asymptotic} by means of a Cole--Hopf transformation.

\noindent$\bullet$ \textbf{Vanishing cell diffusion and a
hyperbolic--elliptic limit.}
We next fix $D=1$ and consider
\[
\delta=\epsilon\to0.
\]
The formal limit is the nonlocal hyperbolic--elliptic system
\beq
\label{df}
\partial_tu+\chi\nabla\cdot\big(u(1-u)\nabla v\big)=0,
\qquad
-\Delta v+v=u,
\qquad
u(0,x)=u^0(x).
\eeq
In this regime, the cell diffusion disappears and the
equation changes from parabolic to hyperbolic. Discontinuous solutions
may occur, so an entropy condition is needed to select the appropriate
weak solution. For the general theory of scalar conservation laws, we
refer to \cite{dafermos2005hyperbolic,serre1999systems}.

In one space dimension, the vanishing cell-diffusion limit was studied
in \cite{dolak2005keller} by using a uniform $BV$ estimate. In
\cite{bp2009}, entropy solutions of a hyperbolic Keller--Segel model
were constructed by passing to the limit in a parabolic approximation
through a kinetic formulation.

In dimensions $n>1$, a uniform $BV$ estimate for the density is not
available, and direct $L^1$ compactness cannot be used; see
\cite{PERTHAME2014}. We instead pass to the limit in a kinetic
formulation. We prove that the limiting kinetic function is a
characteristic function. This kinetic reduction gives the strong
convergence of $u_\epsilon$ and allows us to identify the nonlinear
chemotactic flux.

\noindent$\bullet$ \textbf{Strong attraction, small chemical diffusion,
and a static sharp-interface limit.}
Finally, we consider the simultaneous scaling
\[
\chi=\epsilon^{-1}\to\infty,
\qquad
D=\epsilon^2\to0,
\qquad
\delta=1.
\]
The exponent $m\ge1$ is fixed. In this regime, the chemotactic attraction
becomes strong while the chemical diffusion becomes small. The singular
aggregation term and the loss of uniform elliptic regularity make the
limit problem delicate.

For characteristic initial data of finite perimeter, we obtain a
uniform bound on the singular energy. This bound gives $BV$ compactness
and forces the limiting density to take only the values $0$ and $1$.
The energy dissipation also implies that the limiting flux vanishes.
As a consequence, on every fixed time interval, both $u_\epsilon$ and
$v_\epsilon$ converge to the time-independent initial characteristic
function.

The simultaneous strong-attraction and small-chemical-diffusion limit
was studied in \cite{kim2024density}, where a connection with a
Hele--Shaw free-boundary problem with surface tension was established.
A related incompressible-limit problem was analyzed in
\cite{kim2023density}.
\vspace{2mm}

These three results show that different balances between cell diffusion,
chemotactic attraction, and chemical diffusion lead to different
effective behaviors within the same nonlinear Keller--Segel system.
From a modeling point of view, the three scalings can be interpreted as
different biological regimes. Analytically, all three limits require
strong compactness of the density, but this compactness is obtained by
a different mechanism in each regime.
\vspace{2mm}

The rest of the paper is organized as follows. We first introduce the
energy structure and the weak-solution framework. We then study the
three asymptotic regimes in the order described above. Numerical
simulations and concluding remarks are given at the end.

\begin{notation}
For any fixed $T>0$ and $\Omega\subset\mathbb{R}^n$, we denote\\
$\bullet$ $Q_T=\Omega\times(0,T)$,
$\Gamma_T=\partial\Omega\times(0,T)$.\\
$\bullet$ $C>0$ denotes a generic constant that may depend on
$T$, $\Omega$, $m$, the initial data, and any fixed parameters, but is
independent of the asymptotic parameter $\epsilon$.
\end{notation}

\section{Energy and entropy}
Classically, system \eqref{eq:KSDC} comes with an energy. We set
\beq \label{def: elementary}
\vp_\delta(u)=
\begin{cases}
\displaystyle \delta\int_0^u\frac{w^{m-2}}{1-w}\,dw,&m>1,\\[5pt]
\displaystyle \delta\ln\frac{u}{1-u},&m=1,
\end{cases}
\qquad
\Phi_\delta(u)=\int_0^u\vp_\delta(w)\,dw+a_\delta,
\eeq
where
\[
a_\delta:=-\min_{0\le s\le1}\int_0^s\vp_\delta(w)\,dw.
\]
When $\delta=1$, we write $\vp=\vp_1$ and $\Phi=\Phi_1$. In particular, when $m=1$,
\[
\Phi_\delta(u)=\delta\big[u\ln u+(1-u)\ln(1-u)\big]+\delta\ln2\ge0.
\]

One can check that smooth solutions satisfy the energy equality
\beq \bepa 
\cae (u)= \dis \int_\Omega \big[ \Phi_\delta(u)+ \f\chi{2} u(1-v)] dx,
\\[12pt]
\begin{aligned}
\dis \frac d {dt} \cae (u(t))&= - \int_\Omega u(1-u) |\nabla (\vp_\delta(u) - \chi v)|^2 dx\\
&= - \int_\Omega \frac{1}{u(1-u)} |\delta\nabla \frac {u^m}{m} - \chi u (1-u) \nabla v|^2 dx.    
\end{aligned}
\eepa
\label{def:energy}
\eeq
Indeed, using mass conservation and the self-adjointness of the elliptic operator in \eqref{eq:KSDC}$_2$, we have
\[
\frac{d}{dt}\left(\frac12\int_\Omega u(1-v)\,dx\right)
=-\int_\Omega v\,\partial_tu\,dx.
\]
Furthermore, we know that
\[
\partial_tu=\nabla\cdot\left[u(1-u)\nabla\big(\vp_\delta(u)-\chi v\big)\right],
\]
which yields \eqref{def:energy} after integration by parts.

Since $\Phi_\delta\ge0$, $0\le u,v\le1$, and the energy is nonincreasing, we have
\beq\label{energy_bound}
0\le\cae (u)\leq  \cae (u^0).
\eeq
From this, we conclude that 
\beq
4\int_0^\infty \int_\Omega  |\delta\nabla \frac {u^m}{m} - \chi u (1-u) \nabla v|^2 dx dt \leq  \cae (u^0).
\label{energybound}
\eeq

We can also check that smooth solutions satisfy the entropy equality: 
\beq \bepa
\displaystyle E(u):=\int_\Omega\left[u\ln\frac{u}{M}+(1-u)\ln\frac{1-u}{1-M}\right]dx,
\\[12pt]
\displaystyle \frac{dE(u(t))}{dt}
=\int_\Omega\left[-\vp_\delta'(u)|\nabla u|^2+\chi\nabla u\cdot\nabla v\right]dx.
\eepa
\label{def:entropy}
\eeq
For weak solutions, we use the entropy inequality
\beq\label{entropy_inequality}
E(u(t))-E(u(s))
\le
\int_s^t\!\int_\Omega
\left[-\vp_\delta'(u)|\nabla u|^2
+\chi\nabla u\cdot\nabla v\right]dx\,d\tau,
\qquad \text{for a.e. }0<s<t.
\eeq
The same inequality is assumed for $s=0$, with $E(u(s))$ replaced by
$E(u^0)$.
Since $0<M<1$, the entropy density is nonnegative and bounded on $[0,1]$; hence
\beq\label{entropy_bounded}
0\le E(u)\le E_{\max}.
\eeq


\begin{definition}[Weak solutions]
The pair $(u,v)$ defined on $[0,\infty)\times\Omega$ is a weak solution of system \eqref{eq:KSDC} if, for every $T>0$,
\[
0\le u,v\le1,\qquad u^m\in L^2(0,T;H^1(\Omega)),\qquad
v\in L^2(0,T;H^1(\Omega)),
\]
and, for every $\phi\in C^\infty([0,T]\times\overline\Omega)$ with $\phi(T,\cdot)=0$ and every $\psi\in H^1(\Omega)$,
\beq\label{def_weak1}
\begin{aligned}
&\int_0^T\!\int_\Omega
\left(\delta\nabla\frac{u^m}{m}\cdot\nabla\phi
-\chi u(1-u)\nabla v\cdot\nabla\phi-u\partial_t\phi\right)dxdt
=\int_\Omega u^0(x)\phi(0,x)dx,\\
&\int_\Omega\left(D\nabla v\cdot\nabla\psi+v\psi-u\psi\right)dx=0
\quad\text{for a.e. }t\in(0,T).
\end{aligned}
\eeq
An energy solution is a weak solution satisfying, for a.e.
$t\in(0,T)$,
\[
\cae(u(t))
+\int_0^t\!\int_\Omega
\frac{1}{u(1-u)}
\left|
\delta\nabla\frac{u^m}{m}
-\chi u(1-u)\nabla v
\right|^2 dx\,ds
\le \cae(u^0),
\]
where the fraction is defined to be zero on
$\{u=0\}\cup\{u=1\}$.
\label{weaksolution:D}
\end{definition} 

\section{Limits with \texorpdfstring{$D=\epsilon^2$}{D=epsilon squared}}\label{section_2}
This section first characterizes the stable regime for fixed $D=\epsilon^2$ and then proves convergence, as $D\to0$, to the effective nonlinear diffusion equation.

\subsection{Stable solution with \texorpdfstring{$\delta=1$}{delta=1}}
Thanks to energy dissipation~\eqref{def:energy}, we can characterize sufficiently regular steady states satisfying $0<u<1$ in $\Omega$. With $\vp(u)$ defined by \eqref{def: elementary}, we obtain that 
\beq
\nabla \f{{u}^m}{m} = \chi  u(1-u) \nabla v, \qquad \f{u^{m-2}}{1-u}\nabla u =\nabla \vp(u)= \chi \nabla v.
\label{constanteq}
\eeq
Therefore, there is a constant $\lb$ such that 
\beq
\vp(u) = \chi  (v + \lb).
\label{vp_rule}
\eeq
Then, such steady states are solutions of the following problem: find the pair $(v,\lb)$ such that
\beq \bepa
\begin{aligned}
&-\epsilon^2\Delta v +   v  =  \vp^{-1} (\chi (v+ \lb)),&x\in \Omega,
\\
&\nabla v\cdot \vec{n}=0,&x\in \partial \Omega,\\
&\displaystyle\int_\Omega \vp^{-1}(\chi(v+\lb))\,dx=M.&
\end{aligned}
\eepa
\label{eq:v} \eeq

\begin{theorem}[Uniqueness of solution]\label{th3}
For $1\le m<2$ and $\chi\le\frac{(3-m)^{3-m}}{(2-m)^{2-m}}$, system \eqref{eq:v} has a unique pair $(v,\lb)$, namely
\[
v\equiv M,\qquad \lb=\frac{\vp(M)}{\chi}-M.
\]
Consequently, $u\equiv M$ in \eqref{vp_rule}.
\end{theorem}
\begin{proof}
Multiplying \eqref{eq:v} by $-\Delta v$ and integrating over $\Omega$, we obtain
\[\begin{aligned}
   \epsilon^2\int_{\Omega} |\Delta v |^2 dx +\int_{\Omega} |\nabla v|^2 dx&=\int_{\Omega}  \nabla \vp^{-1} (\chi (v+ \lb))\cdot \nabla v dx\\
   &=\int_{\Omega} {\chi u^{2-m}(1-u)}|\nabla v|^2 dx\\
   &\le \chi \frac{(2-m)^{2-m}}{(3-m)^{3-m}}\int_{\Omega}|\nabla v|^2 dx, 
\end{aligned}
\]
where we use 
\[
\sup_{0\le u\le1}\{u^{2-m}(1-u)\}=\frac{(2-m)^{2-m}}{(3-m)^{3-m}},\qquad 1\le m<2.
\]
Indeed, set $f(u):=u^{2-m}(1-u)$. A direct computation gives $f''(u)<0$ for $0<u<1$ and $f'(\frac{2-m}{3-m})=0$.
Then we have
\[
\epsilon^2\int_{\Omega} |\Delta v |^2 dx +[1-\chi \frac{(2-m)^{2-m}}{(3-m)^{3-m}}]\int_{\Omega} |\nabla v|^2 dx\le 0.
\]
Under the assumption on $\chi$ and the mass conservation, we get
\[
\Delta v=\nabla v=0\Longleftrightarrow v\equiv M.
\]
Consequently, we have
\[
 u=  \vp^{-1} (\chi (v+ \lb))=-\epsilon^2\Delta v +   v \equiv M.
\]
Substituting $u=v=M$ into \eqref{vp_rule} gives $\lb=\vp(M)/\chi-M$.
\end{proof}
\commentout{
\paragraph{Compactness of $v$.} We multiply the second equation of \eqref{eq:KSDC} by $-\Delta v$ and integrate it with Cauchy-
Schwarz inequality
\[
\epsilon^2\int_\Omega |\Delta v |^2 dx +\int_\Omega |\nabla v|^2 dx=-\int_\Omega  u\cdot \Delta v dx \leq\frac{1}{2\epsilon^2}\int_\Omega u^2 dx+\frac{\epsilon^2}{2}\int_\Omega |\Delta v|^2 dx,
\]
which leads to
\[
\frac{\epsilon^2}{2}\int_\Omega |\Delta v|^2 dx+\int_\Omega |\nabla v|^2 dx\leq \frac{1}{2\epsilon^2}\int_\Omega u^2 dx.
\]
Since $(u,v)$ is bounded from \eqref{MaxPrinciple}, it is clear that $\|u \|_{L^2(\Omega)}\leq 1$. Then we have
\beq\label{1}
\int_\Omega |\nabla v|^2 dx\leq \frac{1}{2\epsilon^2}.
\eeq
}

\commentout{
\paragraph{For $\epsilon$ small.}
Let $v=e^{-\frac{s}{\epsilon}}$ be an increasing solution, $s=s(x)\in \mathcal{C}^2(\Omega)$, we can rewrite \eqref{eq:KSDC} as
\[
\epsilon \Delta s-|\nabla s|^2+1=\vp^{-1}(v+\lb)e^{\frac{s}{\epsilon}}.
\]
When $\epsilon\to 0$, we have
\beq\bepa\label{2}
|\nabla s|^2=1,\qquad s>0,\quad v=-\lb,\\
|\nabla s|^2\leq 1,\qquad s=0,\quad v=1.
\eepa\eeq
Then by continuity of $s$, the solution $v(x)$ jumps from $-\lb$ to $1$ at the point $x_0$ such that $s(x_0)=0$.
In particular, we specify \eqref{2} in one dimension (see Figure~\ref{fig:label2})
\[\bepa
s_x=1,\qquad s>0,\quad v=-\lb,\\
s_x=0,\qquad s=0,\quad v=1.
\eepa\]

\begin{figure}[ht]
\centering
\includegraphics[width=6cm,height=4cm]{MQBH/Figure/s.png}
\\[-8pt]
\caption{Sketch of the relations between $v(x)$ and $s(x)$.}
\label{fig:label2}
\end{figure}
}

\subsection{\texorpdfstring{Limits in \eqref{eq:KSDC} in the stable case with $D=\epsilon^2$}{Limits in the stable case}} We first state the limit equation.

\noindent{\bf Limit equation.} 
As $\epsilon\to0$, the elliptic equation formally gives $v=u$. Hence the candidate limit equation is the following scalar equation:

\beq\label{19}
\partial_t u-\Delta \frac{u^m}{m}+\dv [\chi u(1-u)\nabla u]=0.
\eeq
We determine the conditions on $m$ and $\chi$ under which the limit equation~\eqref{19} is formally parabolic. We may rewrite equation \eqref{19} into the parabolic form
\[\bepa
\partial_t u-\Delta Q(u)=0,\\[8pt]
Q(u)=\dis\frac{u^m}{m}-\chi(\frac{u^2}{2}-\frac{u^3}{3}).
\eepa\]
Then we get the essential condition
\beq\label{parabolic}
Q'(u)\ge 0\Longleftrightarrow\chi\le \frac{u^{m-2}}{1-u}=\vp'(u),\qquad 0<u<1.
\eeq
Taking the infimum in \eqref{parabolic}, we obtain
\beq\label{limit condition}
\chi\le \inf_{0<u<1}\vp'(u)=\vp'(\frac{2-m}{3-m})=\frac{(3-m)^{3-m}}{(2-m)^{2-m}},\qquad 1\le m<2.
\eeq

To further study the stability of constant solutions, we mention the following inequalities.

\begin{lemma}\label{ineq}
Let $\lb_1>0$ be the first positive Neumann eigenvalue. If $u\in H^1(\Omega)$, $0\le u\le1$, and $\int_\Omega u\,dx=M$, then there exist constants $c,C>0$, depending only on $M$, such that
\beq\label{ineq2}
c\|u-M\|_{L^2(\Omega)}^2\le E(u)
\le C\|u-M\|_{L^2(\Omega)}^2
\le \frac{C}{\lb_1}\|\nabla u\|_{L^2(\Omega)}^2.
\eeq
\end{lemma}
\begin{proof}
Let
\[
h(s):=s\ln\frac{s}{M}+(1-s)\ln\frac{1-s}{1-M},
\qquad 0\le s\le1.
\]
Since $0<M<1$, we have
\[
h(M)=h'(M)=0,
\qquad
h''(s)=\frac{1}{s(1-s)}>0,\quad 0<s<1.
\]
Hence $h(s)>0$ for $s\neq M$. Moreover, by Taylor's expansion around
$s=M$,
\[
\lim_{s\to M}\frac{h(s)}{(s-M)^2}
=\frac{h''(M)}{2}
=\frac{1}{2M(1-M)}>0.
\]
Therefore, if we define
\[
g(s):=
\begin{cases}
\dfrac{h(s)}{(s-M)^2}, & s\neq M,\\[6pt]
\dfrac{1}{2M(1-M)}, & s=M,
\end{cases}
\]
then \(g\) is continuous and strictly positive on \([0,1]\). Thus, there
exist constants $c,C>0$, depending only on $M$, such that
\[
c(s-M)^2\le h(s)\le C(s-M)^2,
\qquad 0\le s\le1.
\]
Integrating over $\Omega$ gives
\[
c\|u-M\|_{L^2(\Omega)}^2
\le E(u)
\le C\|u-M\|_{L^2(\Omega)}^2.
\]

Finally, since $|\Omega|=1$ and $\int_\Omega u\,dx=M$, we have
$\int_\Omega(u-M)\,dx=0$. Hence, by the Poincar\'e--Wirtinger
inequality,
\[
\|u-M\|_{L^2(\Omega)}^2
\le\frac{1}{\lb_1}\|\nabla u\|_{L^2(\Omega)}^2.
\]

We complete the proof.
\end{proof}

\vspace{2mm}

Now we are ready to show the stability of the constant solution and the exponential decay.

\begin{proposition}[Stability of the constant solution]
Let $1\le m<2$ and
\beq\label{decay_rate}
B(m):=\frac{(3-m)^{3-m}}{(2-m)^{2-m}}-\chi>0.
\eeq
For every $1\le p,q<\infty$, an energy solution $(u,v)$ of system \eqref{eq:KSDC} with $D=\epsilon^2$ and $\delta=1$ satisfying \eqref{entropy_inequality} satisfies
\[
\|(u-M)(t)\|_{L^p(\Omega)}+\|(v-M)(t)\|_{L^q(\Omega)}\le Ce^{-ct},\qquad \text{for a.e. }t\ge0,
\]
where $C,c>0$ may depend on $M$, $\Omega$, $m$, $\chi$, $p$, and $q$, but are independent of $\epsilon$.
\end{proposition}

\begin{proof}
Since $1\le m<2$, we have
\[
\vp'(s)=\frac{s^{m-2}}{1-s}
\ge c_m:=\frac{(3-m)^{3-m}}{(2-m)^{2-m}}>0.
\]
For every fixed $\epsilon>0$, testing \eqref{eq:KSDC}$_2$ by $v$ gives, for a.e. $t>0$,
\[
\epsilon^2\|\nabla v(t)\|_{L^2(\Omega)}^2
+\|v(t)\|_{L^2(\Omega)}^2
=\int_\Omega u(t)v(t)\,dx\le1.
\]
In particular, $\nabla v\in L^2(Q_T)$ for every $T>0$.
Hence, by the entropy inequality and Young's inequality,
\[
\frac{c_m}{2}\int_0^T\!\int_\Omega |\nabla u|^2\,dxdt
\le E(u^0)
+\frac{\chi^2}{2c_m}
\int_0^T\!\int_\Omega|\nabla v|^2\,dxdt<\infty.
\]
Thus,
\[
u\in L^2(0,T;H^1(\Omega)).
\]
For every fixed $\epsilon>0$, elliptic regularity gives
$v(t)\in H^2(\Omega)$ for a.e. $t>0$.
Multiplying \eqref{eq:KSDC}$_2$ by $-\Delta v$ and integrating over $\Omega$, we obtain
\beq
\epsilon^2\|\Delta v\|_{L^2(\Omega)}^2+\|\nabla v\|_{L^2(\Omega)}^2
=\int_\Omega\nabla u\cdot\nabla v\,dx
\le\|\nabla u\|_{L^2(\Omega)}\|\nabla v\|_{L^2(\Omega)}.
\eeq
Using the Poincar\'e--Wirtinger inequality associated with the
Neumann Laplacian,
\beq\label{ineq1}
\int_\Omega |\Delta v|^2dx\ge \lb_1\int_\Omega|\nabla v|^2dx,
\eeq
where $\lb_1$ is the first positive Neumann eigenvalue, we find that
\beq\label{restrain_nabla_v}
\left\|\nabla u \right\|_{L^{2}(\Omega)}
\ge (\epsilon^2\lb_1+1)
\left\|\nabla v \right\|_{L^{2}(\Omega)}.
\eeq

Applying the above inequality to the integrand in \eqref{def:entropy}, for a.e. $\tau>0$, we have
\beq\label{entrpy_decay}
\begin{aligned}
&\int_\Omega\left[-\varphi'(u)|\nabla u|^2
+\chi\nabla u\cdot\nabla v\right]dx\\
&\le
-\int_\Omega \varphi'(u)|\nabla u|^2dx
+\frac{\chi}{\epsilon^2\lb_1+1}
\left\|\nabla u\right\|_{L^2(\Omega)}^2\\
&\le
-B(m)\left\|\nabla u\right\|_{L^2(\Omega)}^2
\le -cE(u),
\end{aligned}
\eeq
where the functions are evaluated at time $\tau$ and $c>0$ is independent of $\epsilon$ by \eqref{ineq2}.
Combining \eqref{entropy_inequality} with \eqref{entrpy_decay}, we obtain
\[
E(u(t))-E(u(s))\le -c\int_s^t E(u(\tau))\,d\tau.
\]
Using also the case $s=0$ and applying Gr\"onwall's inequality, we obtain
\beq\label{et}
E(u(t))\le E(u^0)e^{-ct}.
\eeq
Together with the lower bound in \eqref{ineq2}, this gives
\[
\|u(t)-M\|_{L^2(\Omega)}\le Ce^{-ct/2}.
\]
By virtue of \eqref{MaxPrinciple} and $L^p$ interpolation, for every
$1\le p<\infty$ there exist constants $C,c>0$, possibly depending on $p$, such that
\[
\|u(t)-M\|_{L^p(\Omega)}\le Ce^{-ct}.
\]

Furthermore, multiplying \eqref{eq:KSDC}$_2$ by $v-M$ and integrating over $\Omega$, we obtain
\[
\epsilon^2\|\nabla v\|_{L^2(\Omega)}^2
+\|v-M\|_{L^2(\Omega)}^2
\le
\|u-M\|_{L^2(\Omega)}
\|v-M\|_{L^2(\Omega)},
\]
which implies
\[
\|v-M\|_{L^2(\Omega)}
\le
\|u-M\|_{L^2(\Omega)}.
\]
Using again \eqref{MaxPrinciple} and $L^q$ interpolation, we obtain the exponential decay of
$\|v-M\|_{L^q(\Omega)}$ for every $1\le q<\infty$.
\end{proof}

\subsection{\texorpdfstring{Strong convergence as $\epsilon\to0$}{Strong convergence as epsilon tends to zero}}
For each $\epsilon>0$, let $(u_\epsilon,v_\epsilon)$ be an energy solution in the sense of Definition~\ref{weaksolution:D}. To derive the limit system as $D=\epsilon^2\to0$, we first prove strong convergence and then pass to the limit in \eqref{eq:KSDC}. We state the main result as follows.

\begin{theorem}\label{hks}
Let $1\le m<2$ and $\chi<\frac{(3-m)^{3-m}}{(2-m)^{2-m}}$. Let $(u_\epsilon,v_\epsilon)$ be energy solutions of system \eqref{eq:KSDC} with $D=\epsilon^2$, $\delta=1$, and $u_\epsilon^0=u^0$. We further assume that \eqref{entropy_inequality} holds. Then, up to a subsequence, there exists $u\in L^\infty(Q_T)\cap L^2(0,T;H^1(\Omega))$ such that, for every $1\le p<\infty$,
\[
u_\epsilon\to u,\qquad v_\epsilon\to u\qquad\text{strongly in }L^p(Q_T)\quad\text{as }\epsilon\to0.
\]
Moreover, $u$ is a weak solution of
\beq
\p_t u-\Delta\frac{u^m}{m}
+\dv[\chi u(1-u)\nabla u]=0
\qquad\text{in }Q_T,
\label{baru:equ}
\eeq
with initial datum $u^0$ and the no-flux boundary condition,
in the usual weak sense.
\end{theorem}
\begin{proof}
Because of the boundedness of solutions in~\eqref{MaxPrinciple}, there are subsequences, still denoted by $\left\{u_\epsilon\right\}$ and $\left\{v_\epsilon\right\}$, and functions $u,v\in L^\infty(Q_T)$ such that
\beq \bepa
u_\epsilon\rightharpoonup^{*} u
\quad \text{in }L^\infty(Q_T),
\\
v_\epsilon\rightharpoonup^{*} v
\quad \text{in }L^\infty(Q_T).
\eepa
\label{weaklystar}
\eeq

We now verify that the solution $(u_\epsilon,v_\epsilon)$ converges to $(u,v)$ strongly in $L^2\big(Q_T \big)$ as $\epsilon \rightarrow 0$. 
On the one hand, it follows from the energy dissipation~\eqref{energybound} that 
\beq
\|\p_t u_\epsilon\|_{L^2(0,T;H^{-1}(\Omega))}\le\left\|\nabla \frac {(u_\epsilon)^m}{m} - \chi u_\epsilon (1-u_\epsilon) \nabla v_\epsilon \right\|_{L^2(Q_T )}\le C.
\label{partial_t:u}
\eeq
On the other hand, for $1\le m<2$, it follows from the entropy inequality and \eqref{entrpy_decay} that the entropy functional defined in~\eqref{def:entropy} satisfies
\beq\label{19_11}\begin{aligned}
E(u_\epsilon(T))-E(u_\epsilon^0)
&\le\iint_{Q_T}\left[-\vp'(u_\epsilon)|\nabla u_\epsilon|^2
+\chi\nabla u_\epsilon\cdot\nabla v_\epsilon\right]dxdt\\
&\le -C\|\nabla u_\epsilon\|_{L^2(Q_T)}^2,
\end{aligned}
\eeq
where $C>0$ is independent of $\epsilon$. Since $E(u_\epsilon(T))\ge0$, we obtain
\beq\label{3}
\|\nabla u_\epsilon\|_{L^2(Q_T)}^2
\le \frac{E(u_\epsilon^0)-E(u_\epsilon(T))}{C}
\le \frac{E(u_\epsilon^0)}{C}.
\eeq

Therefore, from \eqref{partial_t:u} and \eqref{3}, the Aubin--Lions lemma first gives, up to a subsequence,
\[
u_\epsilon\to u\qquad\text{strongly in }L^2(Q_T).
\]
Since $0\le u_\epsilon,u\le1$, this convergence also holds for every $1\le p<\infty$, namely,
\beq
u_\epsilon\rightarrow u\qquad\text{in }L^{p}\big(Q_T\big)\quad \text{as }\epsilon\rightarrow0.
\label{strongcon}
\eeq

By \eqref{restrain_nabla_v} and \eqref{3}, we get the boundedness of $\nabla v_\epsilon$, that is,
\[
\|\nabla v_\epsilon\|_{L^2(Q_T)}
\le C\|\nabla u_\epsilon\|_{L^2(Q_T)}
\le C.
\]
Moreover, \eqref{3} yields, up to the same subsequence,
\[
u_\epsilon\rightharpoonup u
\qquad\text{weakly in }L^2(0,T;H^1(\Omega)).
\]
In particular, $u\in L^2(0,T;H^1(\Omega))$, so that $v_\epsilon-u$ is an admissible test function in the elliptic equation.
Multiplying \eqref{eq:KSDC}$_2$ by $(v_\epsilon-u)$ and integrating in $Q_T$, we obtain
\[
\begin{aligned}
&\epsilon^2\|\nabla v_\epsilon\|_{L^2(Q_T)}^2
-\epsilon^2\iint_{Q_T}\nabla v_\epsilon\cdot\nabla u\,dxdt
+\|v_\epsilon-u\|_{L^2(Q_T)}^2\\
&\qquad
=\iint_{Q_T}(u_\epsilon-u)(v_\epsilon-u)\,dxdt.
\end{aligned}
\]
By Young's inequality,
\[
\begin{aligned}
\frac{\epsilon^2}{2}\|\nabla v_\epsilon\|_{L^2(Q_T)}^2
+\frac12\|v_\epsilon-u\|_{L^2(Q_T)}^2
\le
\frac{\epsilon^2}{2}\|\nabla u\|_{L^2(Q_T)}^2
+\frac12\|u_\epsilon-u\|_{L^2(Q_T)}^2.
\end{aligned}
\]
Therefore, by \eqref{strongcon}, we first obtain $v_\epsilon\to u$ strongly in $L^2(Q_T)$. Since $0\le v_\epsilon,u\le1$, the convergence also holds for every $1\le p<\infty$:
\beq
v_\epsilon\to u
\qquad\text{in }L^p(Q_T),\quad 1\le p<\infty,
\quad\text{as }\epsilon\to0.
\label{strongconvergence1}
\eeq
Moreover, the uniform gradient bounds and the strong convergence imply, up to the same subsequence,
\[
\nabla u_\epsilon^m\rightharpoonup\nabla u^m,
\qquad
\nabla v_\epsilon\rightharpoonup\nabla u
\quad\text{in }L^2(Q_T).
\]
Thus, by \eqref{strongcon} and \eqref{strongconvergence1}, we can pass to the limit in the full weak formulation \eqref{def_weak1}, including its initial term, and conclude that $u$ is a weak solution satisfying \eqref{baru:equ}.
\end{proof}

\section{\texorpdfstring{Strong convergence when $\delta=\epsilon$}{Strong convergence when delta equals epsilon}}\label{section3}
We fix $D=1$ and set $\delta=\epsilon$. As $\epsilon\to0$, we prove that $(u_\epsilon,v_\epsilon)$ converges strongly to a kinetic entropy solution of the hyperbolic--elliptic system \eqref{hss2}. 
Throughout this section, we consider energy solutions obtained as limits of the
approximation problem \eqref{eq:appro} as $\nu\to0$.
\begin{theorem}\label{Theorem_7}
Let $m\ge1$ and $\chi>0$. For $0<\epsilon\le1$, let $(u_\epsilon,v_\epsilon)$ be energy solutions of system \eqref{eq:KSDC} with $D=1$, $\delta=\epsilon$, and $u_\epsilon^0=u^0$. Then, for every $1<q<\infty$,
\beq\label{T7_1}
\sqrt{\epsilon}\,\|\nabla u_\epsilon^m\|_{L^2(Q_T)}
+\|v_\epsilon\|_{L^\infty(0,T;W^{2,q}(\Omega))}
+\|\partial_tv_\epsilon\|_{L^2(0,T;H^1(\Omega))}\le C_q.
\eeq
Up to a subsequence, there exist $u\in L^\infty(Q_T)$ and $v\in L^\infty(0,T;W^{2,q}(\Omega))$ for every $1<q<\infty$, such that, for every $1\le p<\infty$,
\begin{equation}\label{stronglimit}
u_\epsilon\to u\qquad\text{strongly in }L^p(Q_T)\quad\text{as }\epsilon\to0,
\end{equation}
\begin{equation}\label{ss1}
v_\epsilon\to v\qquad\text{strongly in }L^p(0,T;W^{1,p}(\Omega))\quad\text{as }\epsilon\to0.
\end{equation}
The limit satisfies
\begin{equation}\label{hss2}
\begin{cases}
\partial_tu+\chi\nabla\cdot\big(u(1-u)\nabla v\big)=0,&\text{in }\mathcal D'(Q_T),\\
-\Delta v+v=u,&\text{a.e. in }Q_T,\\
\nabla v\cdot\vec n=0,&\text{on }\Gamma_T,\\
u(0,\cdot)=u^0,&\text{in the kinetic trace sense}.
\end{cases}
\end{equation}
The pair $(u,v)$ is understood as a kinetic entropy solution in the sense of the kinetic formulation below.
\end{theorem}

Before proving Theorem~\ref{Theorem_7}, we first give some elementary estimates.
\begin{lemma}[Elementary estimates]\label{cms}
For $m\ge1$ and $\chi>0$, the energy solution $(u_\epsilon,v_\epsilon)$ of system \eqref{eq:KSDC} with $0<\epsilon\le1$ satisfies 
\begin{equation}\label{e2}
\sqrt{\epsilon}\|\nabla u_\epsilon^m\|_{L^2(Q_T)}\leq C,
\end{equation}
\begin{equation}\label{e3}
\|v_\epsilon\|_{L^\infty(0,T;W^{2,q}(\Omega))}\leq C,\qquad1<q<\infty,
\end{equation}
\begin{equation}\label{e4}
\|\partial_t v_\epsilon\|_{L^2(0,T;H^{1}(\Omega))}\leq C.
\end{equation}

Moreover, for any $T>0$, there exist $u\in L^\infty(Q_T)$ and $v\in L^\infty(0,T;W^{2,q}(\Omega))$ for every $1<q<\infty$ such that, up to a subsequence, as $\epsilon\to0$,
\begin{equation}\label{e5}
u_\epsilon\rightharpoonup^{*}u,\qquad
v_\epsilon\rightharpoonup^{*}v
\quad\text{in }L^\infty(Q_T),
\end{equation}
\begin{equation}\label{e6}
v_\epsilon\to v,\quad \text{in }L^p(0,T;W^{1,p}(\Omega))\quad p\in[1,\infty),
\end{equation}
\begin{equation}\label{e7}
\epsilon \nabla u_\epsilon^m\to 0,\quad \text{in }L^2(Q_T).
\end{equation}
\end{lemma}
\begin{proof}
The $W^{2,q}$ regularity theorem for $-\Delta z+z=f$ with homogeneous Neumann boundary condition gives, for every $1<q<\infty$,
\[
\|v_\epsilon\|_{W^{2,q}(\Omega)}\le C\|u_\epsilon\|_{L^q(\Omega)}\le C,
\]
which proves \eqref{e3}. The weak limits in \eqref{e5} follow directly from \eqref{MaxPrinciple}.
Now we show \eqref{e2}. Multiplying $\eqref{eq:KSDC}_1$ by $m u_\epsilon^m$ and integrating over $Q_T$, we obtain
\beq\label{ming}\begin{aligned}
\epsilon\iint_{Q_T}|\nabla u_\epsilon^m|^2dxdt
&=\frac{m}{m+1}\left(\int_\Omega u_\epsilon^{m+1}(0)dx-\int_\Omega u_\epsilon^{m+1}(T)dx\right)\\
&\quad+m\chi\iint_{Q_T}u_\epsilon(1-u_\epsilon)\nabla u_\epsilon^m\cdot\nabla v_\epsilon\,dxdt\\
&\le C+m^2\chi\iint_{Q_T}
\nabla\left(\frac{u_\epsilon^{m+1}}{m+1}
-\frac{u_\epsilon^{m+2}}{m+2}\right)
\cdot\nabla v_\epsilon\,dxdt\\
&=C-m^2\chi\iint_{Q_T}
\left(\frac{u_\epsilon^{m+1}}{m+1}
-\frac{u_\epsilon^{m+2}}{m+2}\right)
\Delta v_\epsilon\,dxdt\\
&=C+m^2\chi\iint_{Q_T}
\left(\frac{u_\epsilon^{m+1}}{m+1}
-\frac{u_\epsilon^{m+2}}{m+2}\right)
(u_\epsilon-v_\epsilon)\,dxdt\le C.
\end{aligned}\eeq
The last bound follows from $0\le u_\epsilon,v_\epsilon\le1$ and $T<\infty$, and it implies \eqref{e2}.

To obtain \eqref{e4}, we observe from \eqref{energybound} that 
\beq
\|\p_t u_\epsilon\|_{L^2(0,T;H^{-1}(\Omega))}\le
\left\|\epsilon\nabla\frac{u_\epsilon^m}{m}-\chi u_\epsilon(1-u_\epsilon)\nabla v_\epsilon\right\|_{L^2(Q_T)}\le C.
\label{partial_t:u_1}
\eeq
Differentiating \eqref{eq:KSDC}$_2$ with respect to $t$, we have
\begin{equation*}
-\Delta \partial_t v_\epsilon+\partial_t v_\epsilon=\partial_t u_\epsilon
=\nabla\cdot\left(\epsilon\nabla\frac{u_\epsilon^m}{m}-\chi u_\epsilon(1-u_\epsilon)\nabla v_\epsilon\right)
\in L^2(0,T;H^{-1}(\Omega)).
\end{equation*}
Testing this equation by $\partial_t v_\epsilon$ yields
\[
\|\partial_t v_\epsilon\|_{L^2(0,T;H^1(\Omega))}
\le C\|\partial_t u_\epsilon\|_{L^2(0,T;H^{-1}(\Omega))},
\]
which implies \eqref{e4}.
Using the Aubin--Lions lemma, \eqref{e6} follows.

From \eqref{e2}, we have
\[ \|\epsilon\nabla u_\epsilon^m\|_{L^2(Q_T)}= \sqrt{\epsilon}\cdot\sqrt{\epsilon} \|\nabla u_\epsilon^m\|_{L^2(Q_T)}\le \sqrt{\epsilon} C,
\] which gives \eqref{e7}.
\end{proof}

\subsection{Kinetic Formulation}
In our case, the aggregation flux $\chi u_\epsilon(1-u_\epsilon)\nabla v_\epsilon$ contains a quadratic dependence on the density, so the weak convergence of $u_\epsilon$ is not sufficient to pass to the limit in this nonlinear term and strong convergence of the density is required. Following the kinetic method in \cite{bp2009}, we derive a kinetic formulation for the weak limit and identify it as a characteristic function, which yields the strong convergence of $u_\epsilon$.

For the weak solution $u_\epsilon$ of system \eqref{eq:KSDC}, the corresponding kinetic function $f_\epsilon$ is defined by 
\begin{equation}\label{f_m}
f_\epsilon(t,x,\xi)=\mathbbm{1}_{\xi<u_\epsilon(t,x)},\quad t\in [0,\infty),\ x\in\Omega,\ \xi\in \mathbb{R}_+.
\end{equation}
We extend $f_\epsilon$ by $1$ for $\xi<0$ and by $0$ for $\xi>1$.
Since $0\le f_\epsilon(t,x,\xi)\le1$, there exists $f\in L^\infty(Q_T\times\mathbb R)$ with $0\le f\le1$ such that, up to a subsequence,     
\begin{equation}\label{f_weak}
f_\epsilon\rightharpoonup^* f\quad\text{in }L^\infty(Q_T\times\mathbb R),\qquad\text{as }\epsilon\to0.
\end{equation}

The main result of this section is as follows:

\begin{proposition}\label{strong convergence of u}
Let $u_\epsilon$ be a weak solution of system \eqref{eq:KSDC} obtained as a limit of the approximation problem \eqref{eq:appro}, and let $u$ be its weak-star limit as $\epsilon\to0$. Then the limiting kinetic function satisfies 
\begin{equation}\label{f_def}
f=\mathbbm{1}_{\xi<u},\quad \text{a.e. in } Q_T\times \mathbb{R}_+.
\end{equation}
\end{proposition}

\vspace{2mm}
\noindent{\bf Derivation of the kinetic formulation.} To prove Proposition \ref{strong convergence of u}, 
we first introduce the kinetic formulation corresponding to Eq.~\eqref{eq:KSDC}, namely, the weak solution $u_\epsilon$ satisfies the following formulation (see for instance~\cite{dafermos2005hyperbolic, bp2002})
\begin{equation}\label{kf}
\partial_t f_\epsilon+\chi(\xi-v_\epsilon)g(\xi)\partial_\xi f_\epsilon
+\chi g'(\xi)\nabla v_\epsilon\cdot\nabla_x f_\epsilon
+\frac{\epsilon}{m}\nabla_x\cdot\big(\partial_\xi f_\epsilon\nabla u_\epsilon^m\big)
=\partial_\xi M_\epsilon,
\end{equation}
in $\mathcal D'(Q_T\times\mathbb R)$, where $g(\xi)=\xi(1-\xi)$. 
Here $M_\epsilon\ge0$ is a bounded measure, and \eqref{kf} holds in the
distributional sense in $Q_T\times\mathbb R$, together with
\[
f_\epsilon(0,x,\xi)
=\mathbbm{1}_{\{\xi<u^0(x)\}}.
\]

To apply the chain rule,
we consider the 
approximation problem (see, e.g., \cite{chenperthame2003}). For fixed $\epsilon>0$ and
$\nu>0$, let $(u_{\epsilon,\nu},v_{\epsilon,\nu})$ solve
\beq\label{eq:appro}
\begin{cases}
\partial_tu_{\epsilon,\nu}
-\dfrac{\epsilon}{m}\Delta u_{\epsilon,\nu}^m
-\nu\Delta u_{\epsilon,\nu}
+\chi\nabla\cdot\big(g(u_{\epsilon,\nu})\nabla v_{\epsilon,\nu}\big)=0,\\[4pt]
-\Delta v_{\epsilon,\nu}+v_{\epsilon,\nu}=u_{\epsilon,\nu},
\end{cases}
\eeq
with boundary conditions
\[
\left(\frac{\epsilon}{m}\nabla u_{\epsilon,\nu}^m
+\nu\nabla u_{\epsilon,\nu}
-\chi g(u_{\epsilon,\nu})\nabla v_{\epsilon,\nu}\right)\cdot\vec n=0,
\qquad
\nabla v_{\epsilon,\nu}\cdot\vec n=0,
\]
and smooth initial data $u_\nu^0\in C^\infty(\overline\Omega)$ satisfying the usual compatibility condition,
$0\le u_\nu^0\le1$, and
$u_\nu^0\to u^0$ strongly in $L^1(\Omega)$.
By the maximum principle,
$0\le u_{\epsilon,\nu},v_{\epsilon,\nu}\le1$. 
Let $\eta_\kappa(\cdot,\xi)\in C^2(\mathbb R)$, $\kappa>0$, be convex
functions such that
\[
\eta_\kappa(r,\xi)\to (r-\xi)_+,
\qquad
\partial_r\eta_\kappa(r,\xi)
\to \mathbbm{1}_{\{\xi<r\}}
\]
as $\kappa\to0$, and
\[
\partial_{rr}\eta_\kappa(r,\xi)
\rightharpoonup \delta_{\xi=r}
\]
in the sense of distributions.

The estimates are uniform in $\nu$ and yield
$u_{\epsilon,\nu}^m$ bounded in $L^2(0,T;H^1(\Omega))$,
$\sqrt{\nu}\,\nabla u_{\epsilon,\nu}$ bounded in $L^2(Q_T)$, and
$\partial_tu_{\epsilon,\nu}$ bounded in $L^2(0,T;H^{-1}(\Omega))$.
Hence, by the nonlinear Aubin--Lions compactness theorem
\cite[Theorem~1]{Mou16}, up to a subsequence,
\[
u_{\epsilon,\nu}\to u_\epsilon
\quad\text{strongly in }L^p(Q_T),\qquad 1\le p<\infty.
\]
Moreover,
\[
\nabla u_{\epsilon,\nu}^m\rightharpoonup\nabla u_\epsilon^m
\quad\text{weakly in }L^2(Q_T),
\qquad
\nu\nabla u_{\epsilon,\nu}\to0
\quad\text{strongly in }L^2(Q_T),
\]
and elliptic regularity gives
\[
v_{\epsilon,\nu}\to v_\epsilon
\quad\text{strongly in }L^2(0,T;H^1(\Omega)).
\]
Let
$f_{\epsilon,\nu}:=\mathbbm{1}_{\{\xi<u_{\epsilon,\nu}\}}$. Then
\[
\int_{\mathbb R}|f_{\epsilon,\nu}-f_\epsilon|\,d\xi
=|u_{\epsilon,\nu}-u_\epsilon|,
\]
so that $f_{\epsilon,\nu}\to f_\epsilon$ strongly in
$L^1(Q_T\times\mathbb R)$.

For fixed $\nu>0$, we apply the chain rule to the first equation of
\eqref{eq:appro} with the entropy $\eta_\kappa$ and then let
$\kappa\to0$. Differentiating the resulting entropy relation with
respect to $\xi$ in the distributional sense gives
\begin{equation}\label{kf_nu}
\begin{aligned}
\partial_t f_{\epsilon,\nu}
&+\chi(\xi-v_{\epsilon,\nu})g(\xi)\partial_\xi f_{\epsilon,\nu}
+\chi g'(\xi)\nabla v_{\epsilon,\nu}\cdot\nabla_x f_{\epsilon,\nu}\\
&+\frac{\epsilon}{m}\nabla_x\cdot
\big(\partial_\xi f_{\epsilon,\nu}\nabla u_{\epsilon,\nu}^m\big)
+\nu\nabla_x\cdot
\big(\partial_\xi f_{\epsilon,\nu}\nabla u_{\epsilon,\nu}\big)
=\partial_\xi M_{\epsilon,\nu},
\end{aligned}
\end{equation}
where
\[
M_{\epsilon,\nu}
:=\delta_{\xi=u_{\epsilon,\nu}}
\left(
\frac{4\epsilon}{(m+1)^2}
\left|\nabla u_{\epsilon,\nu}^{(m+1)/2}\right|^2
+\nu|\nabla u_{\epsilon,\nu}|^2
\right)\ge0.
\]
The corresponding initial kinetic datum is
\[
f_{\epsilon,\nu}(0,x,\xi)
=\mathbbm{1}_{\{\xi<u_\nu^0(x)\}}.
\]
Multiplying the first equation of \eqref{eq:appro} by
$u_{\epsilon,\nu}$ and integrating over $Q_T$, we obtain
\[
\begin{aligned}
\|M_{\epsilon,\nu}\|_{\mathcal M(Q_T\times\mathbb R)}
&=\frac{4\epsilon}{(m+1)^2}
\iint_{Q_T}\left|\nabla u_{\epsilon,\nu}^{(m+1)/2}\right|^2\,dxdt
+\nu\iint_{Q_T}|\nabla u_{\epsilon,\nu}|^2\,dxdt\\
&=\frac12\int_\Omega
\big((u_\nu^0)^2-u_{\epsilon,\nu}^2(T)\big)\,dx+\chi\iint_{Q_T}
\left(\frac{u_{\epsilon,\nu}^2}{2}
-\frac{u_{\epsilon,\nu}^3}{3}\right)
(u_{\epsilon,\nu}-v_{\epsilon,\nu})\,dxdt
\le C_T,
\end{aligned}
\]
where $C_T$ is independent of $\epsilon$ and $\nu$. Hence, up to a
subsequence,
\[
M_{\epsilon,\nu}\rightharpoonup^*M_\epsilon
\qquad\text{in }\mathcal M(Q_T\times[0,1])
\quad\text{as }\nu\to0.
\]
Using the convergences above, we may pass to the limit $\nu\to0$ in
the integrated form of \eqref{kf_nu}. The artificial-viscosity term
vanishes since
\[
\|\nu\nabla u_{\epsilon,\nu}\|_{L^2(Q_T)}
\le\sqrt{\nu}\,
\|\sqrt{\nu}\,\nabla u_{\epsilon,\nu}\|_{L^2(Q_T)}\to0.
\]
We thus obtain \eqref{kf}, together with
\[
f_\epsilon(0,x,\xi)
=\mathbbm{1}_{\{\xi<u^0(x)\}}.
\]

\vspace{2mm}
\paragraph{\bf Passing to the limit.}
The weak-* lower semicontinuity gives
\[
\|M_\epsilon\|_{\mathcal M(Q_T\times\mathbb R_+)}\le C_T
\]
uniformly in $\epsilon$. Hence, up to a subsequence,
\[
M_\epsilon\rightharpoonup^*M
\qquad\text{in }\mathcal M(Q_T\times[0,1]).
\]

In the kinetic formulation \eqref{kf}, we study separately the limits of the two nonlinear terms $\frac{\epsilon}{m}\nabla\cdot(\partial_\xi f_\epsilon\nabla u_\epsilon^m)$ and $\chi g'(\xi)\nabla v_\epsilon\cdot\nabla f_\epsilon$.

\vspace{2mm}
\paragraph{\underline{For $\frac{\epsilon}{m}\nabla\cdot(\partial_\xi f_\epsilon\nabla u_\epsilon^m)$}}
For any smooth test function $\varphi\in C_c^\infty(Q_T\times\mathbb R_+)$, \eqref{e7} gives
\[
\iiint\frac{\epsilon}{m}\nabla\cdot(\partial_\xi f_\epsilon\nabla u_\epsilon^m)\varphi
=\iiint\frac{\epsilon}{m}f_\epsilon\nabla u_\epsilon^m\cdot\nabla_x\partial_\xi\varphi\to0.
\]

\paragraph{\underline{For $\chi g'(\xi)\nabla v_\epsilon\cdot\nabla f_\epsilon$}}
We write
\[\begin{aligned}
\chi g'(\xi)\nabla v_\epsilon\cdot\nabla f_\epsilon
&=\chi\nabla\cdot\big(g'(\xi)\nabla v_\epsilon f_\epsilon\big)
-\chi g'(\xi)\Delta v_\epsilon f_\epsilon\\
&=\chi\nabla\cdot\big(g'(\xi)\nabla v_\epsilon f_\epsilon\big)
+\chi(u_\epsilon-v_\epsilon)g'(\xi)f_\epsilon.
\end{aligned}\]
By the strong convergence of $v_\epsilon$ and $\nabla v_\epsilon$ and the weak-star convergence of $f_\epsilon$,
\[\begin{aligned}
\chi\nabla\cdot(g'(\xi)\nabla v_\epsilon f_\epsilon)&\rightharpoonup
\chi\nabla\cdot(g'(\xi)\nabla v f),\\
\chi v_\epsilon g'(\xi)f_\epsilon&\rightharpoonup\chi v g'(\xi)f.
\end{aligned}\]
Notice that the remaining term $u_\epsilon g'(\xi)f_\epsilon$ is the difficulty due to the absence of a strong limit of density. To deal with it, we recall that $\|u_\epsilon f_\epsilon\|_{L^\infty}$ is bounded, so that there exists a limit function $\rho=\rho(t,x,\xi)\in [0,1]$ such that
\begin{equation}\label{roh_inf}
u_\epsilon f_\epsilon\rightharpoonup^* \rho
\quad\text{in }L^\infty(Q_T\times\mathbb R),
\quad\text{as }\epsilon\to0.
\end{equation}
Passing to the limit as $\epsilon\to0$ yields
\begin{equation}\label{limit2}
\chi g'(\xi)\nabla v_\epsilon\cdot\nabla f_\epsilon\to
\chi g'(\xi)\nabla v\cdot\nabla f+\chi g'(\xi)(\rho-uf)
\quad\text{in }\mathcal D'(Q_T\times{\mathbb R}).
\end{equation}

In addition, the limit $\rho$ can be identified as follows.
From the identity
\[
u_\epsilon f_\epsilon
=\xi f_\epsilon+\int_\xi^\infty f_\epsilon(\eta)\,d\eta,
\]
passing to the limit in the sense of distributions with respect to
$\xi$ gives
\[
-\partial_\xi(\rho-\xi f)=f.
\]
Since $f=\rho=0$ for $\xi>1$, we obtain
\begin{equation}\label{rho_}
\rho(t,x,\xi)-\xi f(t,x,\xi)
=\int_\xi^\infty f(t,x,\eta)\,d\eta.
\end{equation}

The following estimate is given in \cite[Lemma 4.1]{bp2009}, and we omit the details.
\begin{lemma}\label{proved}
For $T>0$, set
\[
C:=\limsup_{\epsilon\to0} \|u_\epsilon\|_{L^\infty(Q_T)}.
\]
Then, with $\rho$ given in \eqref{rho_}, for a.e. $(t,x,\xi)\in Q_T\times\mathbb R$, we have
\[
|\rho(t,x,\xi)-u(t,x)f(t,x,\xi)|\le Cf(t,x,\xi)(1-f(t,x,\xi)).
\]
\end{lemma}

The remaining terms in \eqref{kf} pass to the limit by
\eqref{e6}, \eqref{f_weak}, and the weak-star convergence of
$M_\epsilon$. Hence,
\[
\partial_t f+\chi(\xi-v)g(\xi)\partial_\xi f
+\chi g'(\xi)\nabla v\cdot\nabla_x f+R
=\partial_\xi M,
\qquad
R:=\chi g'(\xi)(\rho-uf).
\]
By Lemma \ref{proved}, $|R|\le Cf(1-f)$.

We also note that the structural properties of $f_\epsilon$ pass to the
limit:
\[
0\le f\le1,\qquad \partial_\xi f\le0,\qquad
f=1\ \text{for }\xi<0,\qquad f=0\ \text{for }\xi>1,
\]
and
\[
u=\int_0^1 f(t,x,\xi)\,d\xi.
\]
Moreover, $M\ge0$ is supported in $0\le\xi\le1$.
Passing to the limit in the kinetic formulation with test functions not vanishing at $t=0$, we obtain
\[
f(0,x,\xi)=\mathbbm{1}_{\{\xi<u^0(x)\}}.
\]
Passing to the limit in the elliptic equation also gives
\[
-\Delta v+v=u,\qquad \nabla v\cdot\vec n=0.
\]
Hence all the assumptions of the following theorem are satisfied.

The following theorem yields Proposition~\ref{strong convergence of u}.
\begin{theorem}[{Adapted from \cite[Theorem~2.2]{bp2009}}]
Consider a weak solution of the kinetic equation
\beq\label{cite}\bepa\begin{aligned}
&\partial_t f+\chi(\xi-v)g(\xi)\partial_\xi f+\chi g'(\xi)\nabla v\cdot\nabla f+R(t,x,\xi)=\partial_\xi M,\\
&M\ge0\text{ is a bounded measure on }[0,T]\times\Omega\times\mathbb R,\\
&f(0,x,\xi)=\mathbbm{1}_{\xi<u_0(x)},\\
&-\Delta v+v=u:=\int_0^\infty f(t,x,\xi)d\xi,\qquad\nabla v\cdot\vec n=0.
\end{aligned}\eepa\eeq
satisfying the following properties:\\
(i) $0\le f(t,x,\xi)\le1$ and $f=1$ for $\xi<0$, $f=0$ for $\xi>1$, $f$ is nonincreasing in $\xi$,\\
(ii) there exists a constant $C>0$ such that $|R|\le Cf(1-f)$ almost everywhere,\\
(iii) the measure $M$ vanishes for $\xi<0$ or $\xi>1$.\\
Then,
\[
f(t,x,\xi)=\mathbbm{1}_{\{\xi<u(t,x)\}}\quad\text{a.e.}
\]
\end{theorem}

\noindent\textbf{Proof of Theorem \ref{Theorem_7}.} Notice that \eqref{T7_1} and \eqref{ss1} have been proved in Lemma~\ref{cms}. By Proposition~\ref{strong convergence of u}, $f=\mathbbm{1}_{\xi<u}$. 
Hence $f(1-f)=0$ a.e., and the estimate
$|R|\le Cf(1-f)$ gives $R=0$ a.e. Therefore, the limit is a kinetic entropy solution of \eqref{hss2}. Since $0\le u_\epsilon\le1$,
\[
u_\epsilon=\int_0^1 f_\epsilon\,d\xi,\qquad u_\epsilon^2=2\int_0^1\xi f_\epsilon\,d\xi.
\]
The weak-star convergence of $f_\epsilon$ therefore gives $u_\epsilon\rightharpoonup u$ in $L^2(Q_T)$ and
\[
\int_{Q_T}u_\epsilon^2\,dx\,dt\to\int_{Q_T}u^2\,dx\,dt.
\]
Hence $u_\epsilon\to u$ strongly in $L^2(Q_T)$, and the uniform $L^\infty$ bound yields strong convergence in $L^p(Q_T)$ for every $1\le p<\infty$. Using \eqref{ss1} and \eqref{e7}, we pass to the limit in \eqref{def_weak1} and obtain \eqref{hss2}. $\hfill\square$ 

\section{\texorpdfstring{Static sharp-interface limit as $D\to0$ and $\chi\to\infty$}{Static sharp-interface limit}}\label{section4}
We now consider system~\eqref{eq:KSDC} under the scaling $\chi=\epsilon^{-1}$, $D=\epsilon^2$, and $\delta=1$, in which case the system reads
\beq \bepa
\begin{aligned}
&\p_t u_\epsilon - \Delta \f{(u_\epsilon)^m}{m} + \dv [ \epsilon^{-1} u_\epsilon(1-u_\epsilon) \nabla v_\epsilon ]=0,
\\
&- \epsilon^{2} \Delta  v_\epsilon +  v_\epsilon  =  u_\epsilon.
\end{aligned}
\eepa
\label{eq:KSDC11}
\eeq 
For $u\in L^2(\Omega)$, let $v_\epsilon[u]\in H^1(\Omega)$ be the weak solution of
\[
-\epsilon^2\Delta v_\epsilon[u]+v_\epsilon[u]=u,
\qquad \nabla v_\epsilon[u]\cdot\vec n=0,
\]
and define
\beq\label{J}
\mathscr J_\epsilon(u):=\frac{1}{2\epsilon}
\int_\Omega u\bigl(1-v_\epsilon[u]\bigr)\,dx.
\eeq

\begin{theorem}\label{Theorem10}
Let $m\ge1$, $\chi=\epsilon^{-1}$, $D=\epsilon^2$, and $\delta=1$. Let $\widetilde\Omega\subset\Omega$ be a nontrivial finite-perimeter set, let $u_\epsilon^0=\mathbbm{1}_{\widetilde\Omega}$, and assume
\[
\sup_{0<\epsilon\le1}\mathscr J_\epsilon(u_\epsilon^0)<\infty.
\]
Here $\mathscr J_\epsilon$ is defined in \eqref{J}.
Let $(u_\epsilon,v_\epsilon)$ be energy solutions of \eqref{eq:KSDC11}. Then, up to a subsequence,
\beq\label{T10_1}
u_\epsilon,v_\epsilon\to u\qquad\text{strongly in }L^\infty((0,T);L^1(\Omega)).
\eeq
The limit is independent of time and
\beq\label{T10_2}
u\in BV(\Omega;\{0,1\}),\qquad u=\mathbbm{1}_{\widetilde\Omega}\quad\text{a.e. in }\Omega.
\eeq
Moreover,
\beq\label{gr}
\int_\Omega|Du|\le4\liminf_{\epsilon\to0}\mathscr J_\epsilon(u_\epsilon(t))\qquad\text{for a.e. }t\in(0,T).
\eeq
\end{theorem}

In the energy \eqref{def:energy}, the functional \eqref{J} is the singular part
\[
\mathscr{J}_\epsilon(u_\epsilon )=\frac{1}{2\epsilon}\int_\Omega u_\epsilon(1-v_\epsilon )\,dx.
\]
It does not behave well when $\epsilon$ is small. However, the energy dissipation provides a uniform bound on $\mathscr{J}_\epsilon(u_\epsilon)$, as shown in Lemma~\ref{boundedness_J}.

Before proving Theorem~\ref{Theorem10}, we mention some preliminary formulas.

\begin{lemma}[Singular-energy bound]\label{boundedness_J}
Under the assumptions of Theorem \ref{Theorem10},
\[
0\le\mathscr J_\epsilon(u_\epsilon(t))\le C\qquad\text{for a.e. }t\ge0,
\]
where $C$ is independent of $\epsilon$.
\end{lemma}
\begin{proof}
For $\delta=1$ and $\chi=\epsilon^{-1}$, we have
\[
\cae_\epsilon(u_\epsilon)=\int_\Omega\Phi(u_\epsilon)dx+\mathscr J_\epsilon(u_\epsilon).
\]
Choose $C_m$ so that $\Phi+C_m\ge0$ on $[0,1]$. The energy inequality corresponding to \eqref{def:energy} yields
\[
\mathscr J_\epsilon(u_\epsilon(t))\le\cae_\epsilon(u_\epsilon(t))+C_m|\Omega|
\le\cae_\epsilon(u_\epsilon^0)+C_m|\Omega|\le C,
\]
where we use the assumption on $\mathscr J_\epsilon(u_\epsilon^0)$ and the boundedness of $\Phi$ on $\{0,1\}$.
\end{proof}

For a.e. $t>0$, thanks to the Neumann boundary conditions, we can express \eqref{J} as
\beq\label{energy1}\begin{aligned}
\mathscr J_\epsilon(u_\epsilon)
&=\frac1{2\epsilon}\int_\Omega\big[u_\epsilon(1-u_\epsilon)+(u_\epsilon-v_\epsilon)^2\big]dx
+\frac\epsilon2\int_\Omega|\nabla v_\epsilon|^2dx\\
&=\frac1{2\epsilon}\int_\Omega\big[(1-u_\epsilon)v_\epsilon^2+u_\epsilon(1-v_\epsilon)^2\big]dx
+\frac\epsilon2\int_\Omega|\nabla v_\epsilon|^2dx.
\end{aligned}\eeq

As a consequence, we can establish the following result.

\begin{lemma}\label{proposition6}
Under the assumptions of Theorem~\ref{Theorem10}, we have
\beq\label{uniform_u}
\|u_\epsilon(t)-u_\epsilon(s)\|_{H^{-1}(\Omega)}
\le C\sqrt{t-s},
\qquad 0\le s\le t\le T.
\eeq
\end{lemma}
\begin{proof}
Since $0\le u_\epsilon\le1$ and, by the equation and
\eqref{energybound},
\[
\partial_tu_\epsilon\in L^2(0,T;H^{-1}(\Omega)),
\]
we have
\[
u_\epsilon\in W^{1,2}(0,T;H^{-1}(\Omega))
\subset C([0,T];H^{-1}(\Omega)).
\]
Hence $u_\epsilon(t)$ is well-defined as an element of $H^{-1}(\Omega)$ for every
$t\in[0,T]$.
For a given test function $\psi\in H^1(\Omega)$, the continuity equation
\eqref{def_weak1} implies
\beq
\int_\Omega u_\epsilon(x,t)\psi(x)\,dx
-\int_\Omega u_\epsilon(x,s)\psi(x)\,dx
=
-\int_s^t\int_\Omega
\left(
\nabla\frac{u_\epsilon^m}{m}
-\epsilon^{-1}u_\epsilon(1-u_\epsilon)\nabla v_\epsilon
\right)\cdot\nabla\psi\,dx\,d\tau.
\eeq
Thus,
\beq
\begin{aligned}
&\left|
\int_\Omega
(u_\epsilon(x,t)-u_\epsilon(x,s))\psi(x)\,dx
\right|
\\
&\quad\le
\left(
\int_s^t\int_\Omega
\left|
\nabla\frac{u_\epsilon^m}{m}
-\epsilon^{-1}u_\epsilon(1-u_\epsilon)\nabla v_\epsilon
\right|^2
\,dx\,d\tau
\right)^{1/2}
\left(
\int_s^t\int_\Omega|\nabla\psi|^2\,dx\,d\tau
\right)^{1/2}
\\
&\quad\le
C\|\psi\|_{H^1(\Omega)}(t-s)^{1/2},
\end{aligned}
\eeq
where the last inequality follows from \eqref{energybound}.
This implies \eqref{uniform_u}.
\end{proof}

Now we are ready to prove Theorem~\ref{Theorem10}.

\vspace{2mm}
\noindent{\bf Proof of Theorem \ref{Theorem10}.} With the above lemmas, we can prove our main result.

\vspace{2mm}
\noindent{\bf Step 1 ($BV$ bound).} As in \cite{kim2024density}, we introduce the function
\[
F(v)=\int_0^v 4\min\{\sigma, 1-\sigma\}d\sigma=
\bepa\begin{aligned}
    &2v^2,\qquad &0\le v\le \frac{1}{2},\\
    &4v-2v^2-1,\qquad &\frac{1}{2}\le v\le 1.
\end{aligned}\eepa
\]
We begin by proving the $BV$ bound on $F(v_\epsilon)$.
For any $\epsilon>0$, we have
\[\begin{aligned}
|\nabla F(v_\epsilon)|\le 4|\nabla v_\epsilon|\min\{v_\epsilon, 1-v_\epsilon\}
&\le2\epsilon^{-1}\min\{v_\epsilon^2, (1-v_\epsilon)^2\}+2\epsilon |\nabla v_\epsilon|^2\\
&\le 2\epsilon^{-1}[(1-u_\epsilon)v_\epsilon^2+u_\epsilon(1-v_\epsilon)^2]+2\epsilon |\nabla v_\epsilon|^2,
\end{aligned}\]
which, together with \eqref{energy1} and Lemma \ref{boundedness_J}, gives, for a.e. $t\ge0$,
\beq\label{def:psi}
\int_\Omega|\nabla F(v_\epsilon(t))|dx\le4\mathscr J_\epsilon(u_\epsilon(t))\le C.
\eeq
This implies that $F(v_\epsilon)$ is uniformly bounded in $L^\infty((0,T);BV(\Omega))$, namely,
\beq\label{L^infy}
F(v_{\epsilon})\in L^\infty((0,T);BV(\Omega)).
\eeq

\noindent{\bf Step 2 (Convergence).} Next, we prove that
\beq\label{85}
F(v_{\epsilon})\to u
\qquad\text{strongly in }L^{\infty}((0,T);L^2(\Omega)),
\eeq
and
\beq\label{77}
u_{\epsilon},v_{\epsilon}\to u
\qquad\text{strongly in }L^{\infty}((0,T);L^1(\Omega)).
\eeq

We use the decomposition
\[
F(v_\epsilon)-u_\epsilon
=\big(F(v_\epsilon)-F(u_\epsilon)\big)
+\big(F(u_\epsilon)-u_\epsilon\big).
\]
Since $F$ is Lipschitz, by \eqref{energy1} we have
\beq\label{4.5}
\operatorname*{ess\,sup}_{0<t<T}
\|F(v_\epsilon)-F(u_\epsilon)\|_{L^2(\Omega)}^2
\le C\operatorname*{ess\,sup}_{0<t<T}
\|v_\epsilon-u_\epsilon\|_{L^2(\Omega)}^2
\le C\epsilon.
\eeq
Moreover, by direct calculation,
\[
|F(s)-s|=
\begin{cases}
s(1-2s),&0\le s\le\frac12,\\
(1-s)(2s-1),&\frac12\le s\le1,
\end{cases}
\le s(1-s).
\]
Hence, we obtain
\beq\label{80}
\begin{aligned}
{\operatorname*{ess\,sup}_{0<t<T}}
\|F(u_\epsilon)-u_\epsilon\|_{L^2(\Omega)}^2
\le C{\operatorname*{ess\,sup}_{0<t<T}}
\int_\Omega u_\epsilon^2(1-u_\epsilon)^2\,dx\le C{\operatorname*{ess\,sup}_{0<t<T}}
\int_\Omega u_\epsilon(1-u_\epsilon)\,dx
\le C\epsilon.
\end{aligned}
\eeq
Therefore, from \eqref{4.5} and \eqref{80}, we conclude
\beq\label{46}
{\operatorname*{ess\,sup}_{0<t<T}}
\|F(v_\epsilon)-u_\epsilon\|_{L^2(\Omega)}
\le C\sqrt{\epsilon}.
\eeq

Lemma \ref{proposition6} implies that $u_\epsilon$ is equicontinuous in $H^{-1}(\Omega)$. Since $0\le u_\epsilon\le1$, it is uniformly bounded in $L^2(\Omega)$, and the compact embedding $L^2(\Omega)\Subset H^{-1}(\Omega)$ gives pointwise relative compactness. Hence, by the Arzel\`a--Ascoli theorem, up to a subsequence,
\[
u_\epsilon\to u
\qquad\text{in }C([0,T];H^{-1}(\Omega)).
\]
Together with \eqref{46} and the continuous embedding $L^2(\Omega)\hookrightarrow H^{-1}(\Omega)$, this yields
\[
F(v_\epsilon)\to u
\qquad\text{strongly in }L^\infty((0,T);H^{-1}(\Omega)).
\]

Using \eqref{L^infy}, the compact embedding of $BV(\Omega)$ into
$L^1(\Omega)$, and the compactness lemma in the appendix, we obtain
\[
F(v_\epsilon)\to u
\qquad\text{strongly in }
L^\infty((0,T);L^1(\Omega)).
\]
Since $0\le F(v_\epsilon),u\le1$,
\[
\|F(v_\epsilon)-u\|_{L^2(\Omega)}^2
\le
\|F(v_\epsilon)-u\|_{L^1(\Omega)},
\]
and therefore
\[
F(v_\epsilon)\to u
\qquad\text{strongly in }
L^\infty((0,T);L^2(\Omega)),
\]
which proves \eqref{85}.

Combining \eqref{46} and \eqref{85}, we obtain
\[
u_\epsilon\to u
\qquad\text{strongly in }
L^\infty((0,T);L^2(\Omega)),
\]
and hence also in
$L^\infty((0,T);L^1(\Omega))$.

Finally, from \eqref{energy1} and Lemma \ref{boundedness_J},
\beq\label{73}
{\operatorname*{ess\,sup}_{0<t<T}}
\|v_\epsilon-u_\epsilon\|_{L^2(\Omega)}^2
\le C\epsilon.
\eeq
Thus,
\[
v_\epsilon\to u
\qquad\text{strongly in }
L^\infty((0,T);L^1(\Omega)),
\]
which proves \eqref{77}.

Moreover, by the lower semicontinuity of the $BV$ seminorm and
\eqref{L^infy}, we have
\[
u\in L^\infty((0,T);BV(\Omega)).
\]
On the other hand, \eqref{energy1} gives
\[
{\operatorname*{ess\,sup}_{0<t<T}}
\int_\Omega u_\epsilon(1-u_\epsilon)\,dx
\le C\epsilon.
\]
Since $u_\epsilon\to u$ strongly in
$L^\infty((0,T);L^1(\Omega))$ and $0\le u_\epsilon,u\le1$, we obtain
\[
u_\epsilon(1-u_\epsilon)\to u(1-u)
\qquad\text{strongly in }
L^\infty((0,T);L^1(\Omega)).
\]
Consequently,
\[
u(1-u)=0\qquad\text{a.e. in }Q_T,
\]
and hence
\[
u(t,x)\in\{0,1\}
\qquad\text{for a.e. }(t,x)\in Q_T.
\]
\vspace{2mm}
\noindent{\bf Step 3 (Liminf property).} To obtain \eqref{gr}, combining \eqref{77} and \eqref{73}, since $F$ is Lipschitz and $F(u(t))=u(t)$, we can calculate, for a.e. $t\in(0,T)$,
\beq\label{4.6}
\begin{aligned}
\|F(v_{\epsilon}(t))-F(u(t))\|_{L^1(\Omega)}
&\le C\|v_{\epsilon}(t)-u(t)\|_{L^1(\Omega)}\\
&\le C\|v_{\epsilon}(t)-u_\epsilon(t)\|_{L^1(\Omega)}
+C\|u_\epsilon(t)-u(t)\|_{L^1(\Omega)}
\to 0.
\end{aligned}
\eeq 
By the lower semicontinuity of the $BV$ norm, we get
\beq
\liminf_{\epsilon\to0}\int_\Omega|\nabla F(v_\epsilon(t))|\,dx
\ge \int_\Omega|Du(t)|.
\eeq
From \eqref{def:psi}, we have
\[
\liminf_{\epsilon\to0}\mathscr{J}_\epsilon(u_\epsilon(t))
\ge \frac14\liminf_{\epsilon\to0}\int_\Omega|\nabla F(v_\epsilon(t))|\,dx
\ge \frac14\int_\Omega|Du(t)|.
\]

\noindent\textbf{Step 4 ($u$ only depends on $x$).} We now prove that $u$ is independent of $t$.
From \eqref{energy1} and the singular-energy bound in Lemma~\ref{boundedness_J}, we have
\beq\label{90}
{\operatorname*{ess\,sup}_{t>0}
\|\sqrt{u_\epsilon(t)(1-u_\epsilon(t))}\|_{L^2(\Omega)}^2
\le C\epsilon.}
\eeq
We define
\[
L_\epsilon(t,x):=
\begin{cases}
\displaystyle
\frac{1}{\sqrt{u_\epsilon(1-u_\epsilon)}}
\left(
\nabla\frac{u_\epsilon^m}{m}
-\chi u_\epsilon(1-u_\epsilon)\nabla v_\epsilon
\right),&0<u_\epsilon<1,\\[10pt]
0,&u_\epsilon\in\{0,1\}.
\end{cases}
\]
This definition is consistent since $\nabla u_\epsilon^m=0$ a.e. on the level sets
$\{u_\epsilon=0\}$ and $\{u_\epsilon=1\}$. The energy inequality corresponding to
\eqref{def:energy} then yields
\beq\label{92}
\nabla\frac{u_\epsilon^m}{m}
-\chi u_\epsilon(1-u_\epsilon)\nabla v_\epsilon
=\sqrt{u_\epsilon(1-u_\epsilon)}\,L_\epsilon,
\qquad
\int_0^\infty\!\int_\Omega|L_\epsilon|^2\,dxdt
\le\cae(u_\epsilon^0)\le C.
\eeq
Consequently, after extraction of a subsequence,
\[
L_\epsilon\rightharpoonup L
\quad\text{weakly in }L^2((0,\infty)\times\Omega),
\]
for some $L\in L^2((0,\infty)\times\Omega)$.
Combining \eqref{90} and \eqref{92}, we obtain
\[
\nabla \frac{u_\epsilon^m}{m}
-\chi u_\epsilon(1-u_\epsilon)\nabla v_\epsilon
=
\sqrt{u_\epsilon(1-u_\epsilon)}\,L_\epsilon
\to0
\quad\text{in }L^1((0,T)\times\Omega),
\]
for every $T>0$.

Therefore, we have
\beq
\dv\left(
\nabla \frac{u_\epsilon^m}{m}
-\chi u_\epsilon(1-u_\epsilon)\nabla v_\epsilon
\right)
\to0,
\quad\text{as }\epsilon\to0,
\eeq
in the sense of distributions. As a consequence,
$\partial_tu_\epsilon\to0$ in the sense of distributions and thus
$u$ does not depend on $t$. Moreover, \eqref{46}, the convergence in $C([0,T];H^{-1}(\Omega))$, and $u_\epsilon(0)=\mathbbm{1}_{\widetilde\Omega}$ imply $u(0)=\mathbbm{1}_{\widetilde\Omega}$. Hence $u=\mathbbm{1}_{\widetilde\Omega}$ a.e. in $\Omega$.

The proof of Theorem \ref{Theorem10} is complete.$\hfill\square$


\begin{remark}
The $\Gamma$-convergence result identifies the sharp interfacial energy associated with $\mathscr J_\epsilon$. It is proved in \cite[Theorem~1.5]{kim2024density} that $\mathscr J_\epsilon$ $\Gamma$-converges in $L^1(\Omega)$ to
\[
\mathscr J_0(u)=
\begin{cases}
\displaystyle\frac14\int_\Omega|Du|,&u\in BV(\Omega;\{0,1\}),\\[5pt]
+\infty,&\text{otherwise}.
\end{cases}
\]
In particular, if $u\in BV(\Omega;\{0,1\})$, there exists a recovery sequence $u_\epsilon\to u$ in $L^1(\Omega)$ such that
\[
\limsup_{\epsilon\to0}\mathscr J_\epsilon(u_\epsilon)
\le\frac14\int_\Omega|Du|.
\]
The recovery-sequence construction reduces to the characteristic-function case established in \cite[Proposition 5.3]{mellet2022gammaconvergencenonlocalperimetersbounded}.
\end{remark}

\section{\texorpdfstring{{Numerical simulations}}{Numerical simulations}}

\begin{figure}[ht]
\centering
\subfigure[$D=1$, $t=0$]{
\includegraphics[width=2.5cm,height=3cm]{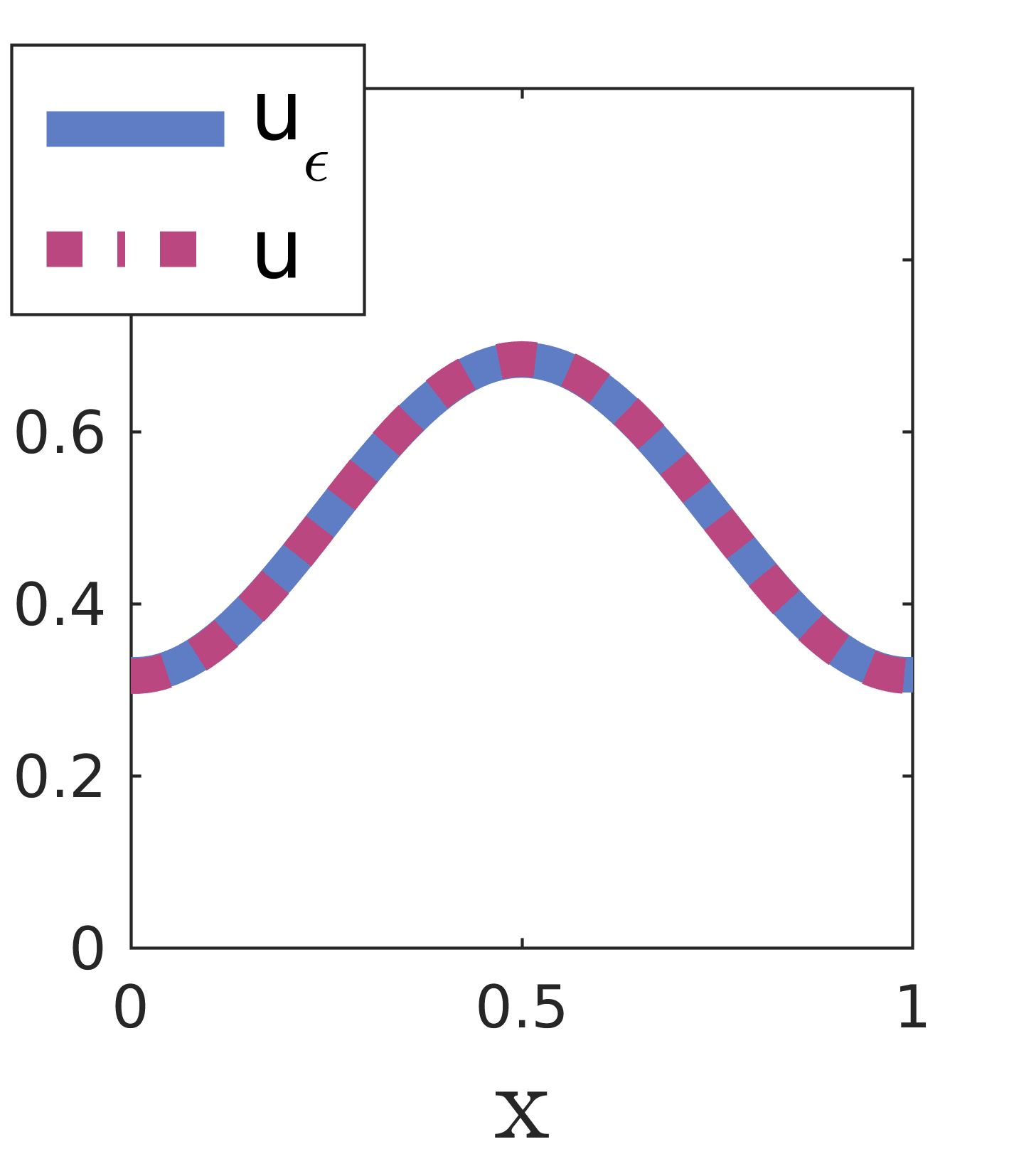}
}\hspace{0.5cm}\subfigure[$D=1$, $t=5$]{
\includegraphics[width=2.5cm,height=3cm]{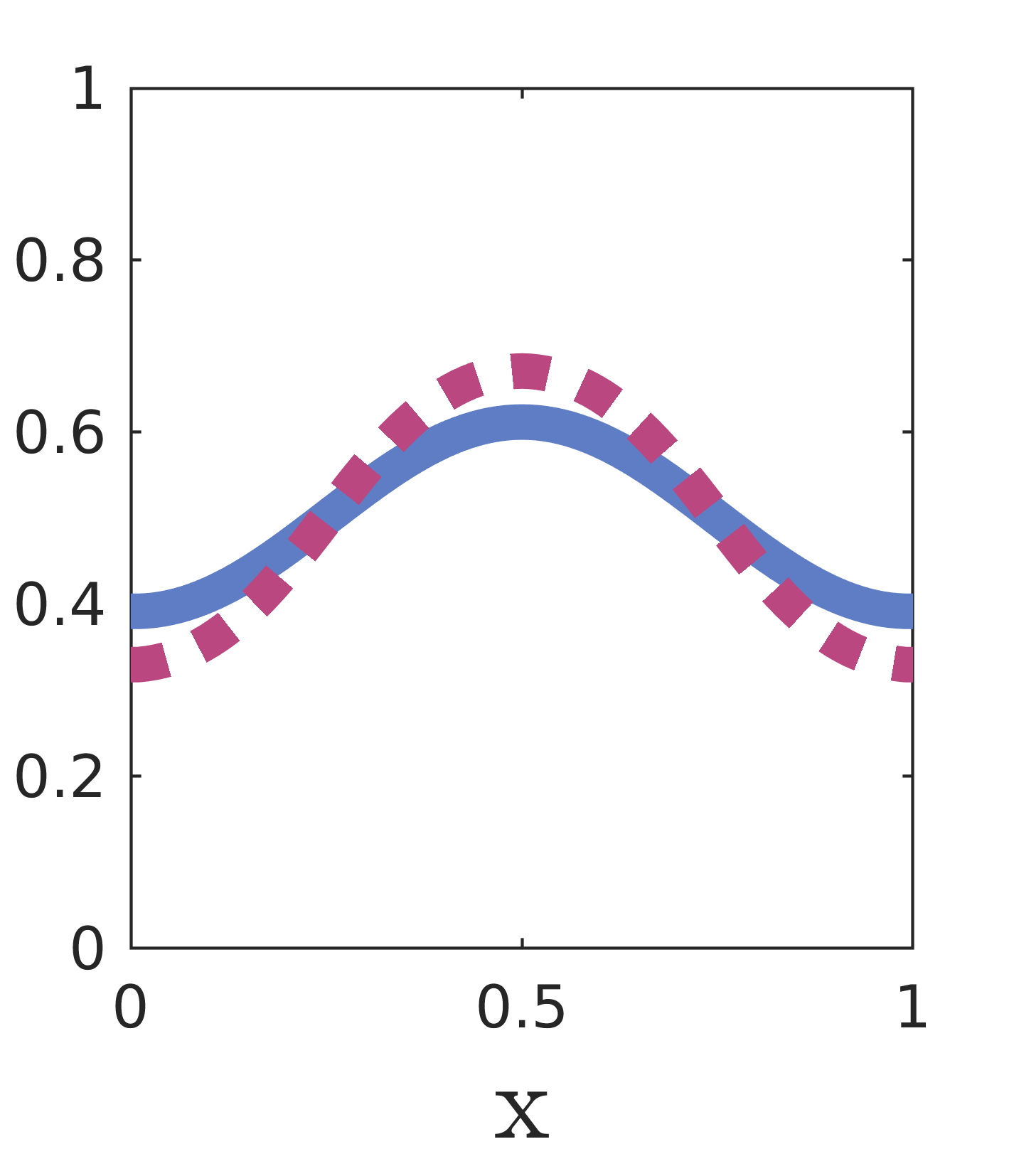}
}\hspace{0.5cm}\subfigure[$D=1$, $t=20$]{
\includegraphics[width=2.5cm,height=3cm]{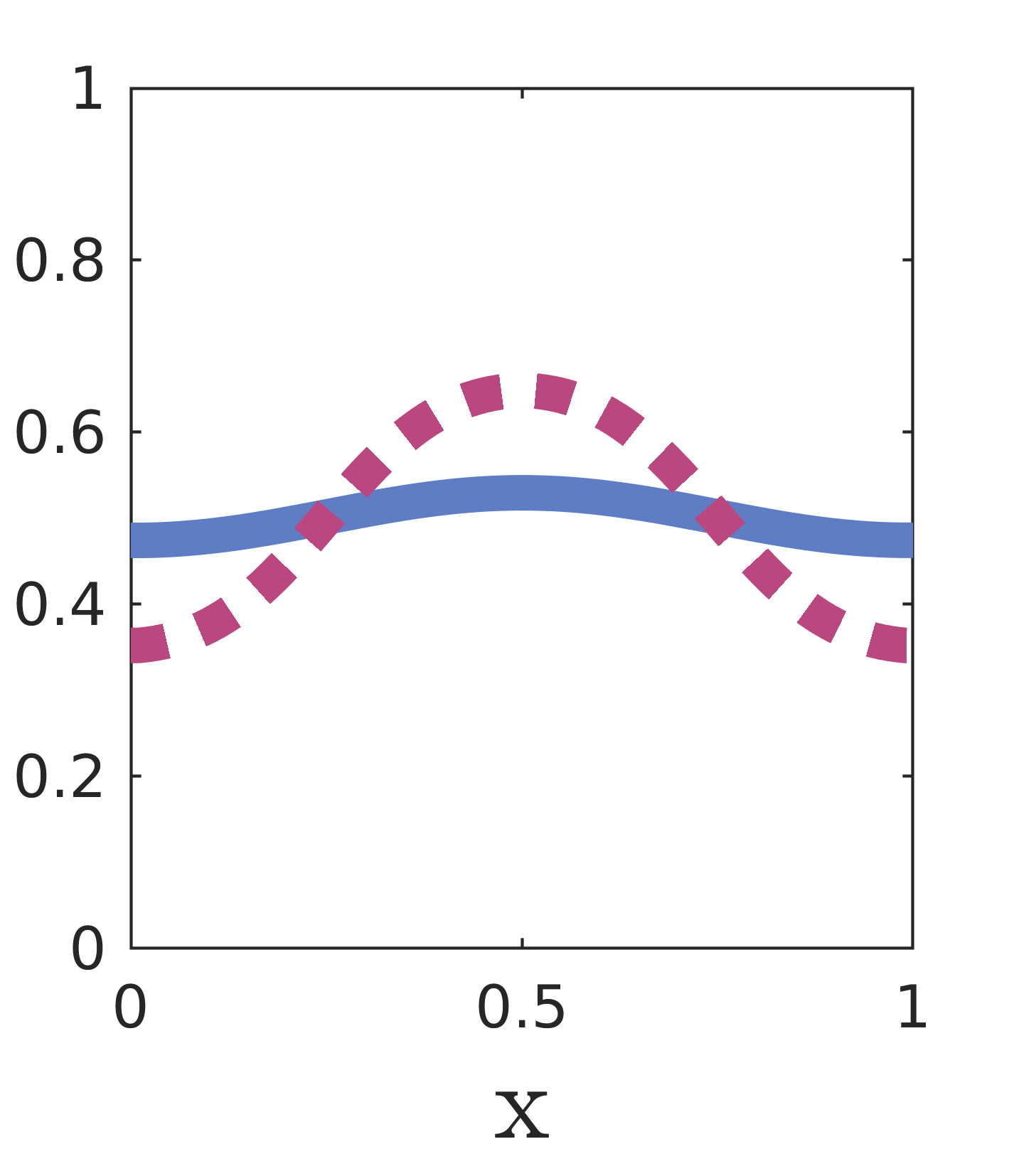}
}\hspace{0.5cm}\subfigure[$D=1$, $t=1000$]{
\includegraphics[width=2.5cm,height=3cm]{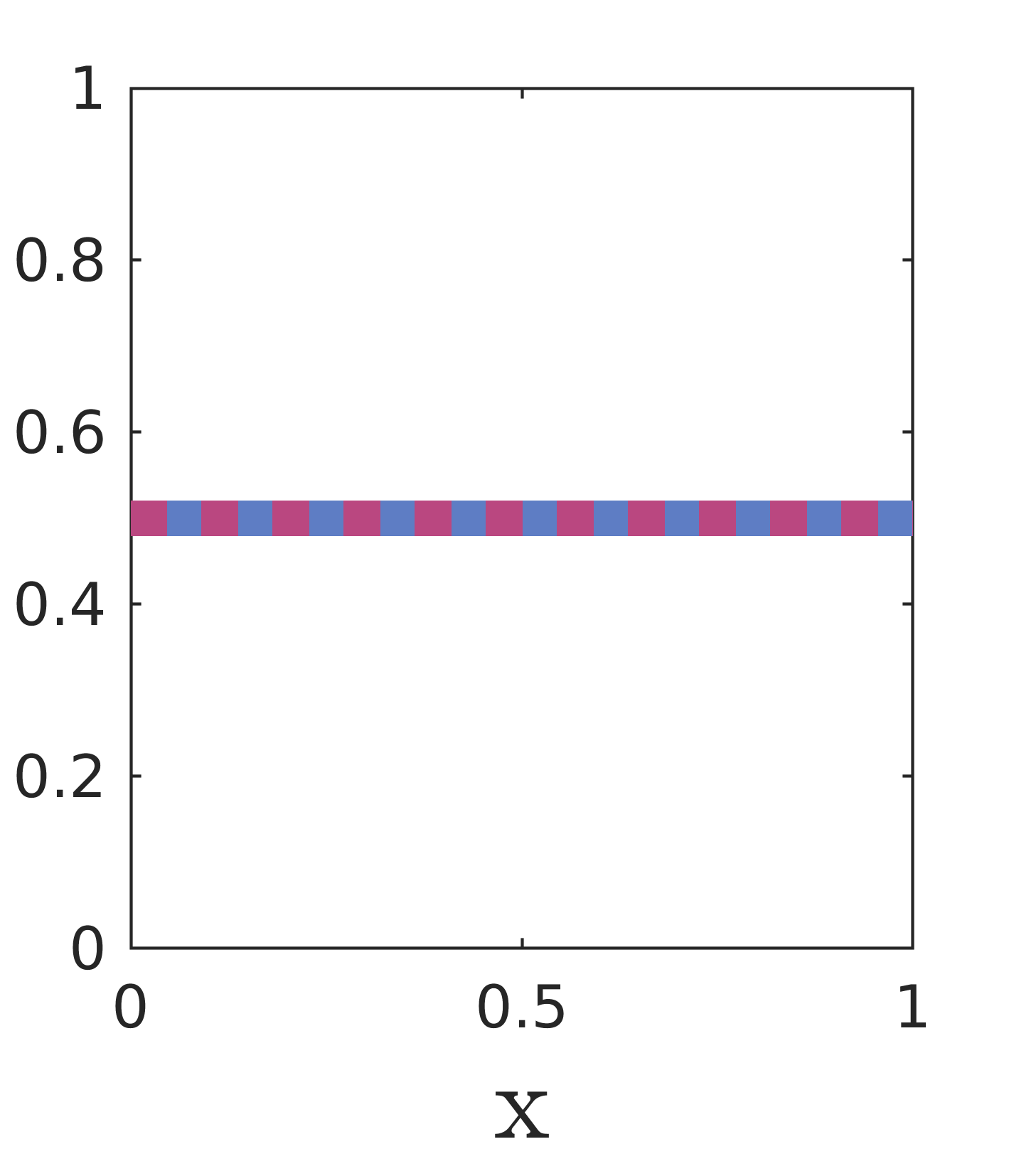}
}\\
\subfigure[$D=0.01$, $t=0$]{
\includegraphics[width=2.5cm,height=3cm]{Figure/u1.png}
}\hspace{0.5cm}\subfigure[$D=0.01$, $t=5$]{
\includegraphics[width=2.5cm,height=3cm]{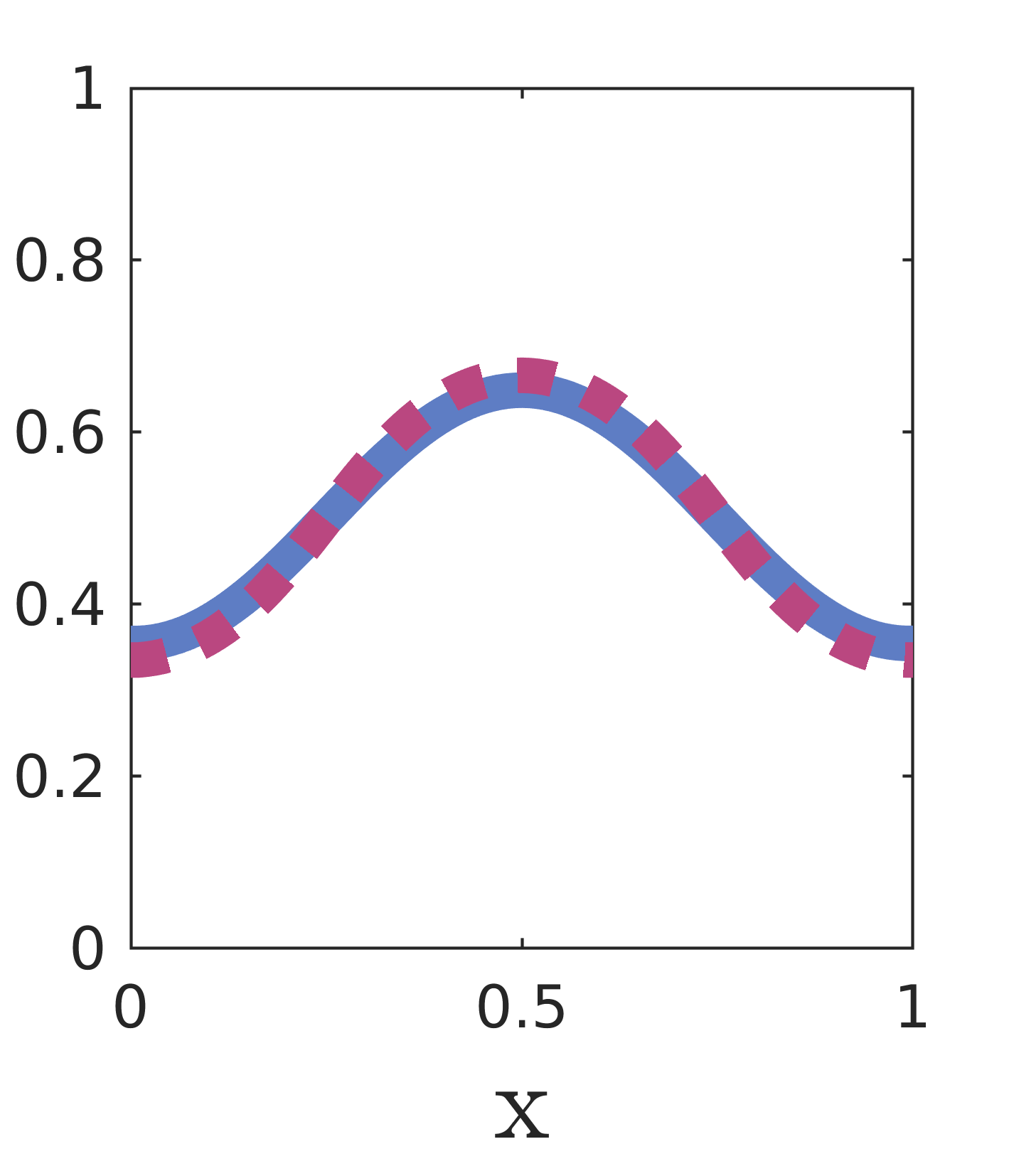}
}\hspace{0.5cm}\subfigure[$D=0.01$, $t=20$]{
\includegraphics[width=2.5cm,height=3cm]{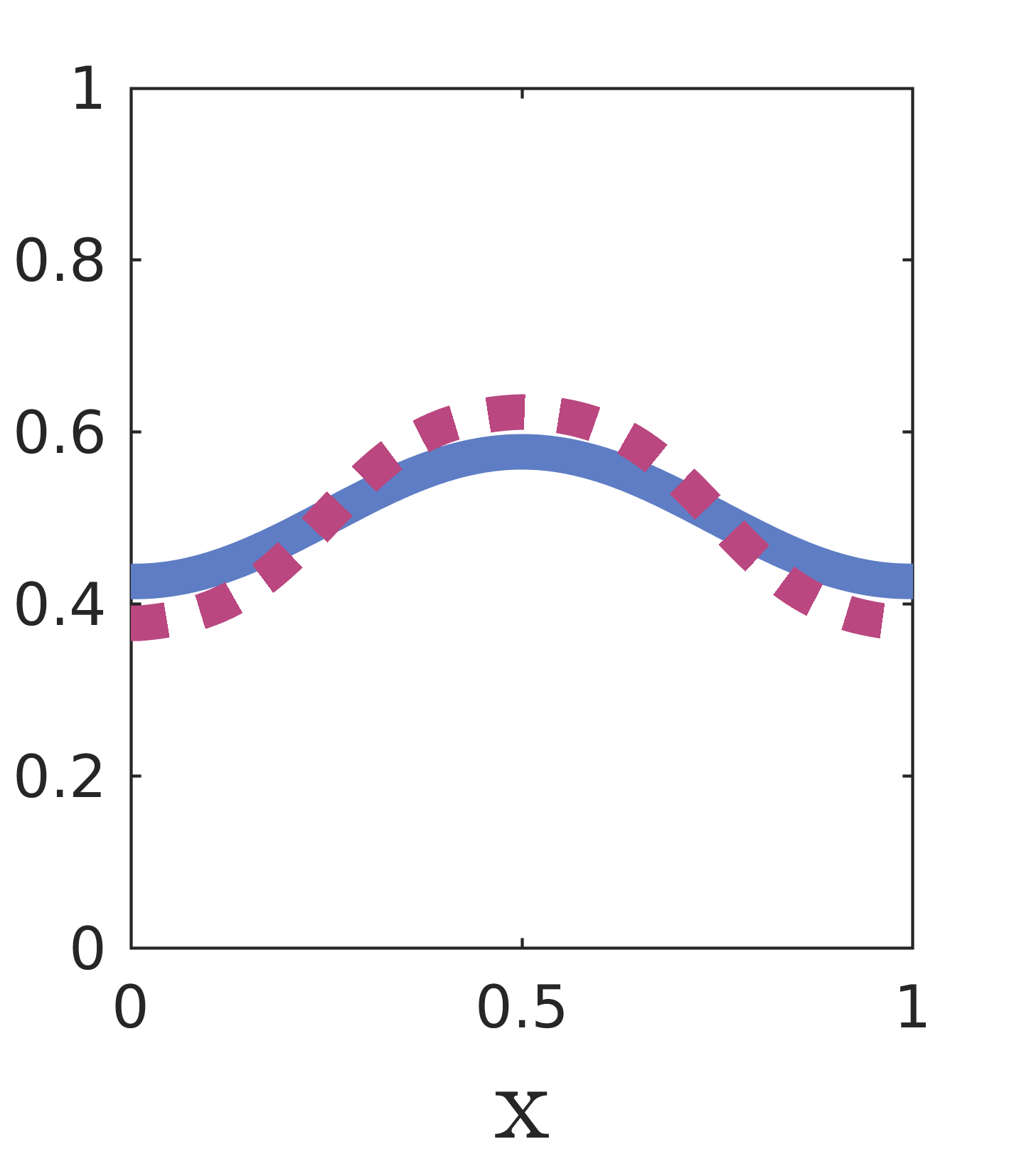}
}\hspace{0.5cm}\subfigure[$D=0.01$, $t=1000$]{
\includegraphics[width=2.5cm,height=3cm]{Figure/u4.png}
}
\\[-8pt]
\caption{Comparison corresponding to Theorem~\ref{hks}: time evolution of $u_\epsilon(t,x)$ (solid line) for \eqref{eq:KSDC} and $u$ (dotted line) for \eqref{baru:equ} in one dimension, with $m=\delta=1$, $\chi=2$, and $M=0.5$.}
\label{fig:label3131}
\end{figure}

\begin{figure}[ht]
\centering
\subfigure[$\delta=100$, $t=0$]{
\includegraphics[width=2.5cm,height=3cm]{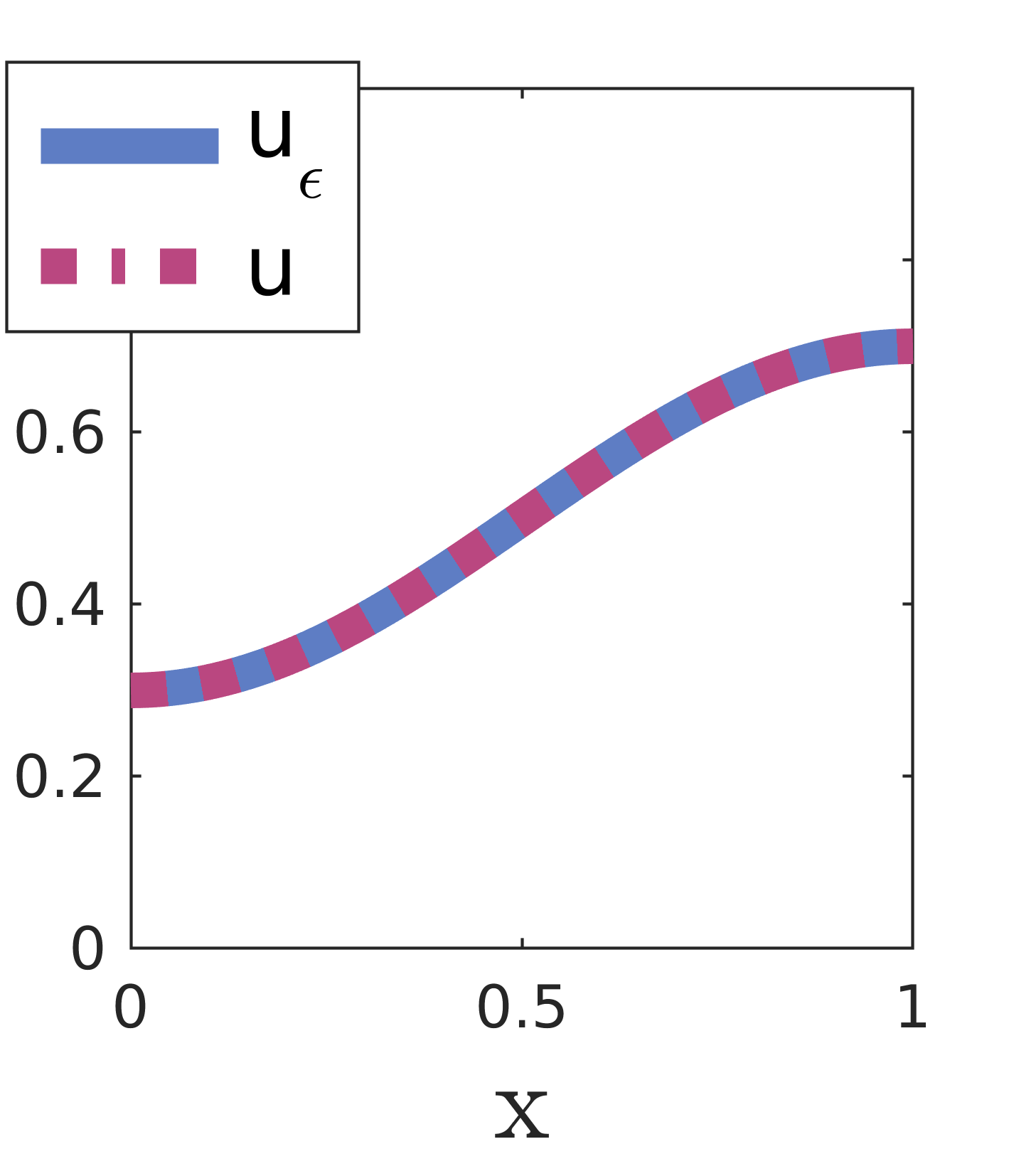}
}\hspace{0.5cm}\subfigure[$\delta=100$, $t=5$]{
\includegraphics[width=2.5cm,height=3cm]{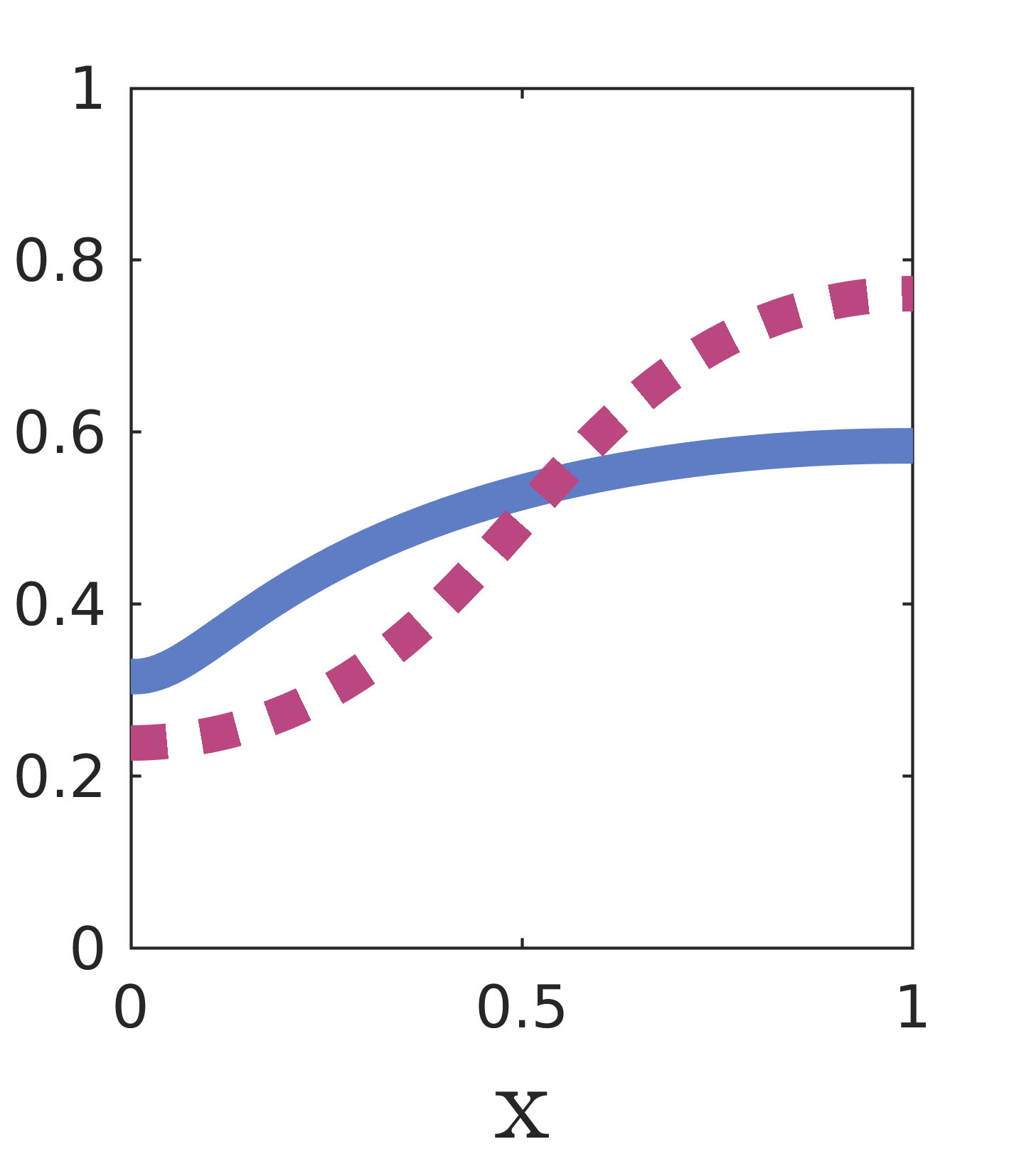}
}\hspace{0.5cm}\subfigure[$\delta=100$, $t=20$]{
\includegraphics[width=2.5cm,height=3cm]{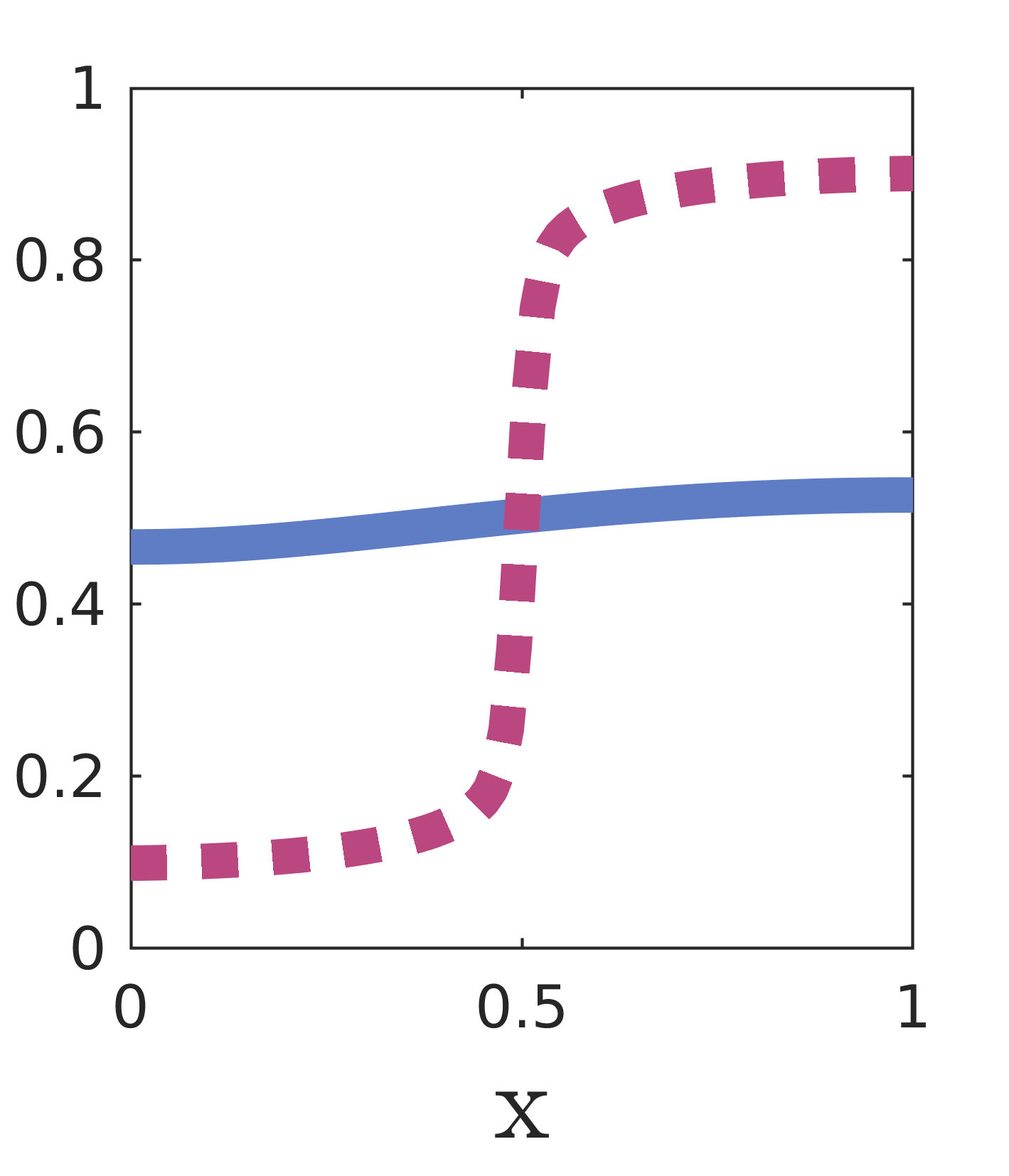}
}\hspace{0.5cm}\subfigure[$\delta=100$, $t=1000$]{
\includegraphics[width=2.5cm,height=3cm]{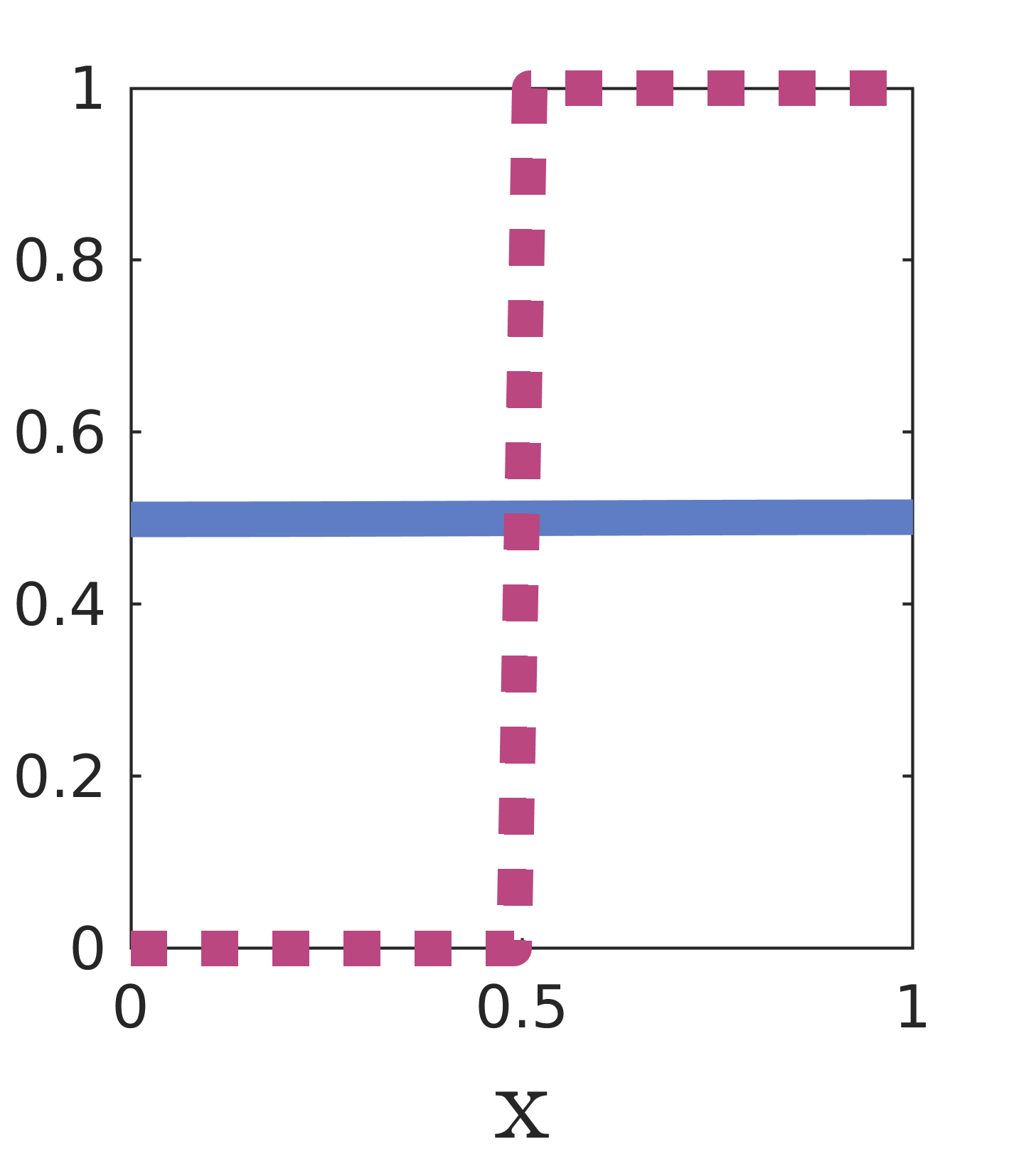}
}\\
\subfigure[$\delta=10$, $t=0$]{
\includegraphics[width=2.5cm,height=3cm]{Figure/case2initial.png}
}\hspace{0.5cm}\subfigure[$\delta=10$, $t=5$]{
\includegraphics[width=2.5cm,height=3cm]{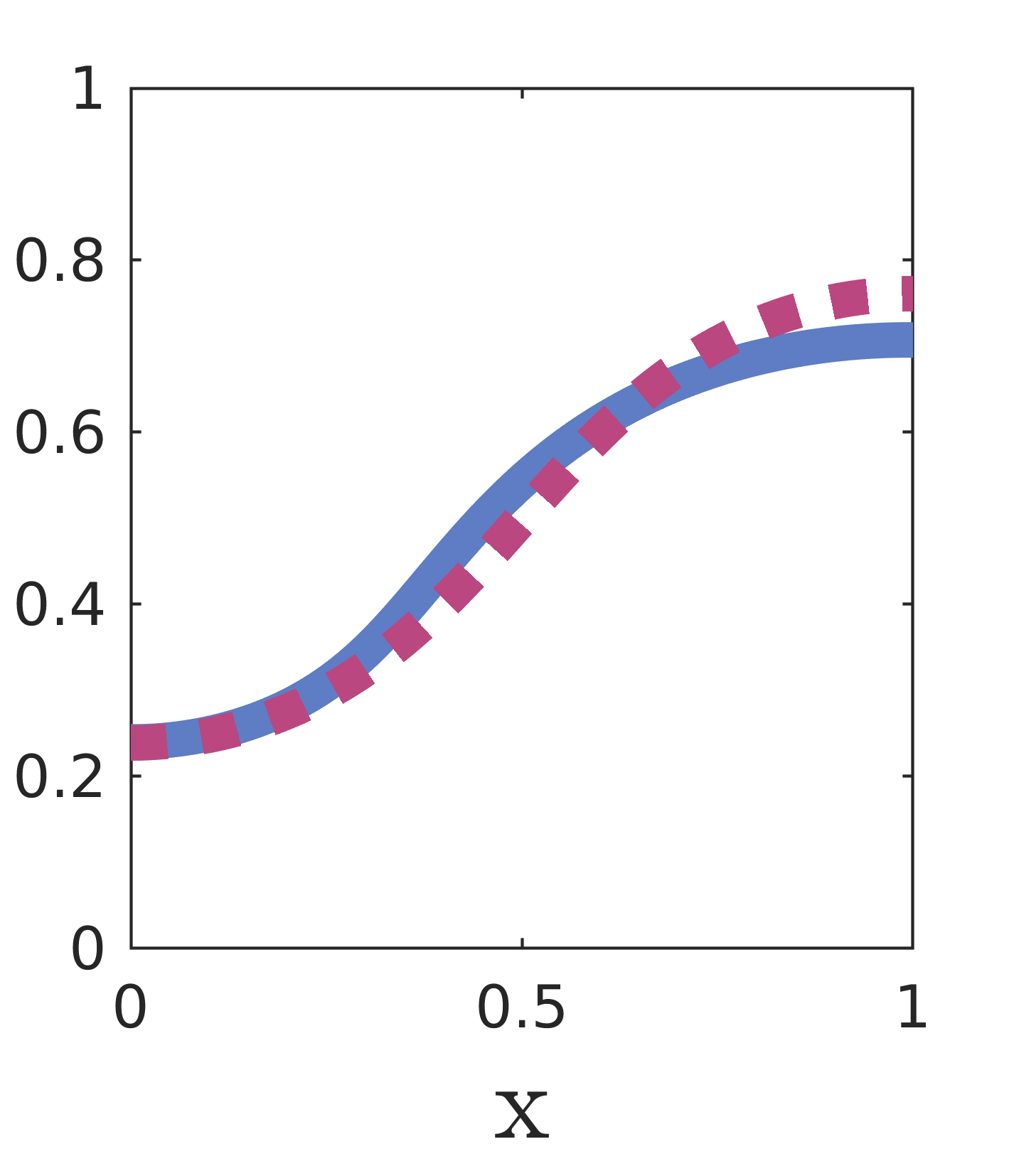}
}\hspace{0.5cm}\subfigure[$\delta=10$, $t=20$]{
\includegraphics[width=2.5cm,height=3cm]{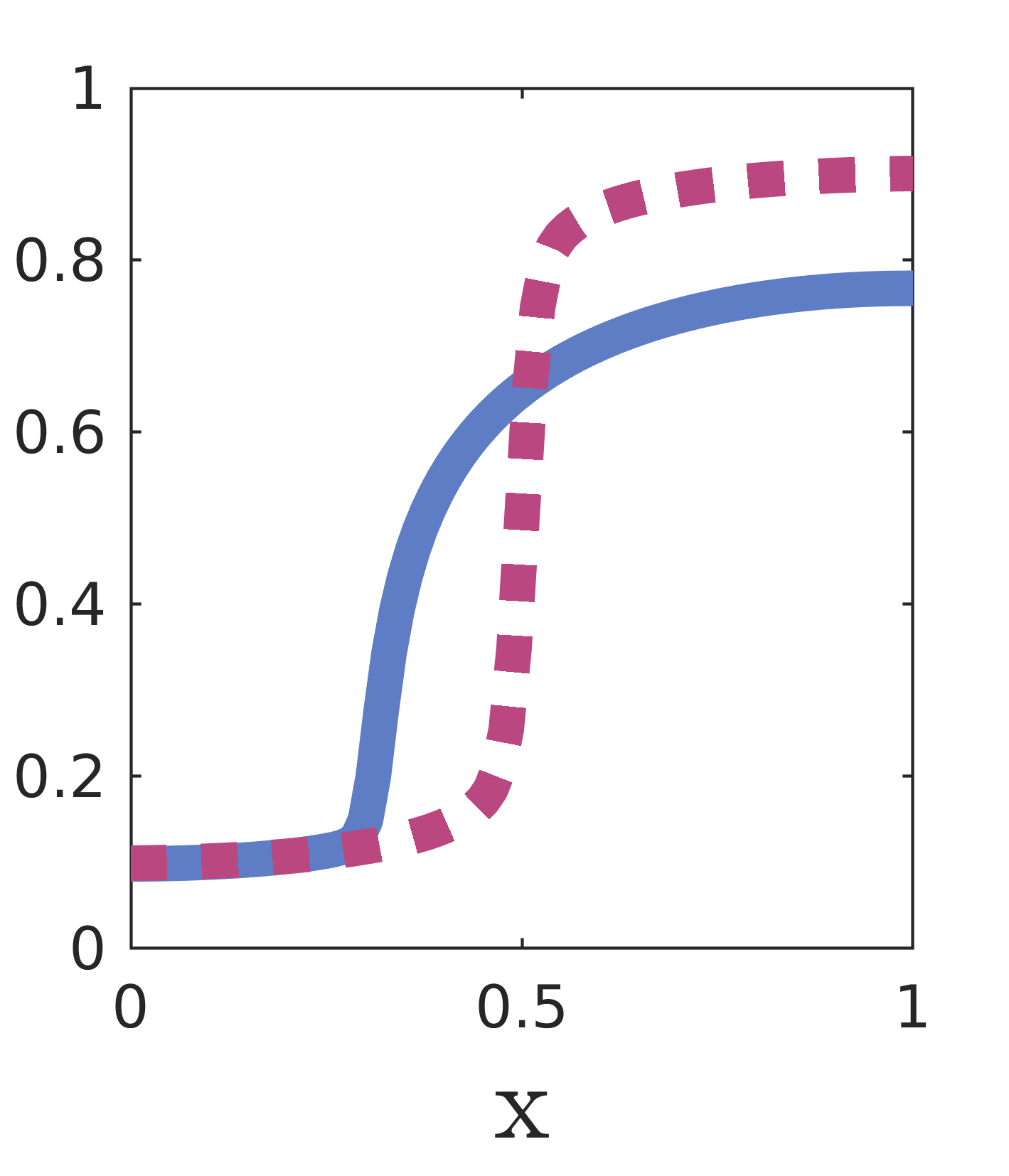}
}\hspace{0.5cm}\subfigure[$\delta=10$, $t=1000$]{
\includegraphics[width=2.5cm,height=3cm]{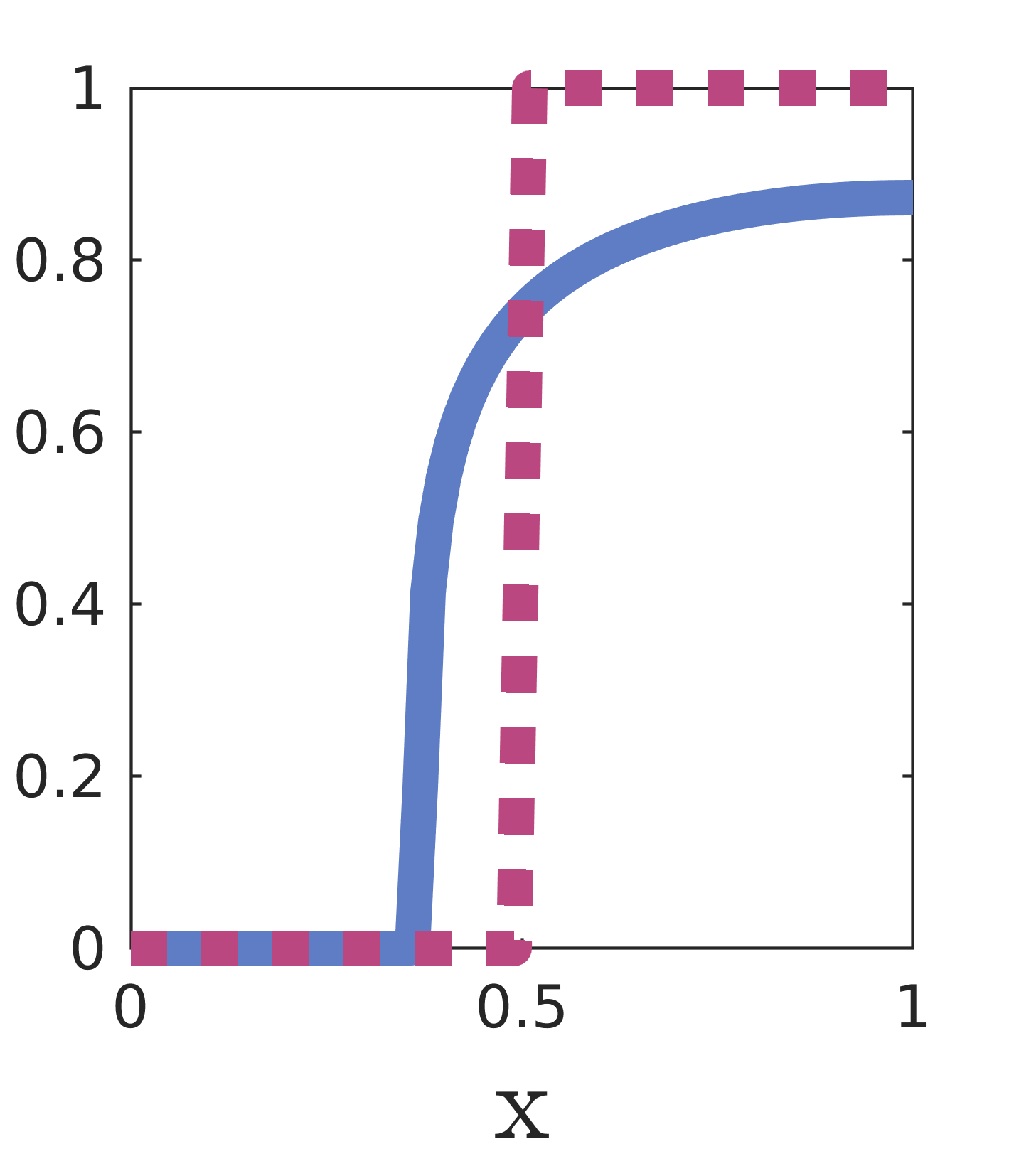}
}\\
\subfigure[$\delta=0.1$, $t=0$]{
\includegraphics[width=2.5cm,height=3cm]{Figure/case2initial.png}
}\hspace{0.5cm}\subfigure[$\delta=0.1$, $t=5$]{
\includegraphics[width=2.5cm,height=3cm]{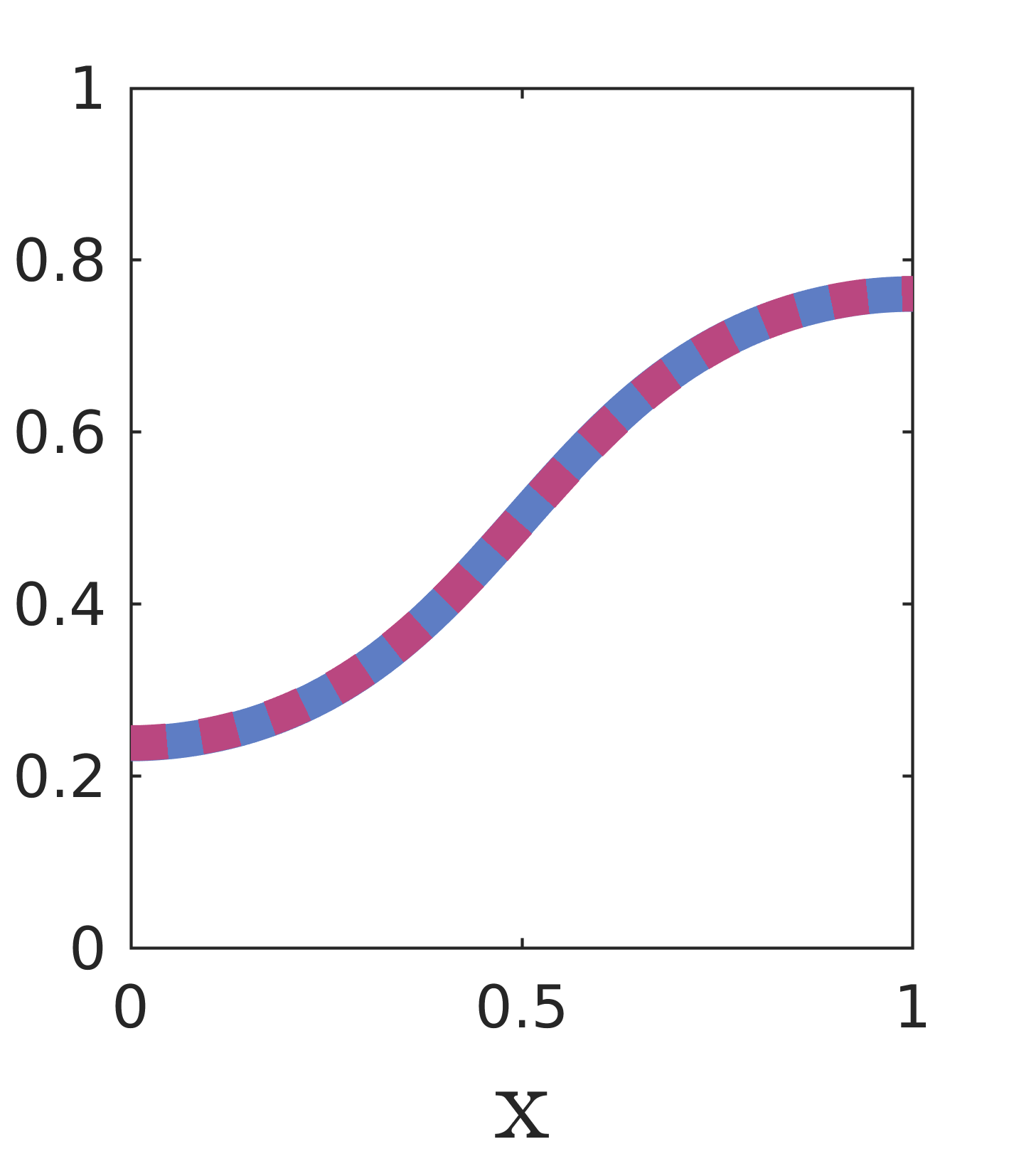}
}\hspace{0.5cm}\subfigure[$\delta=0.1$, $t=20$]{
\includegraphics[width=2.5cm,height=3cm]{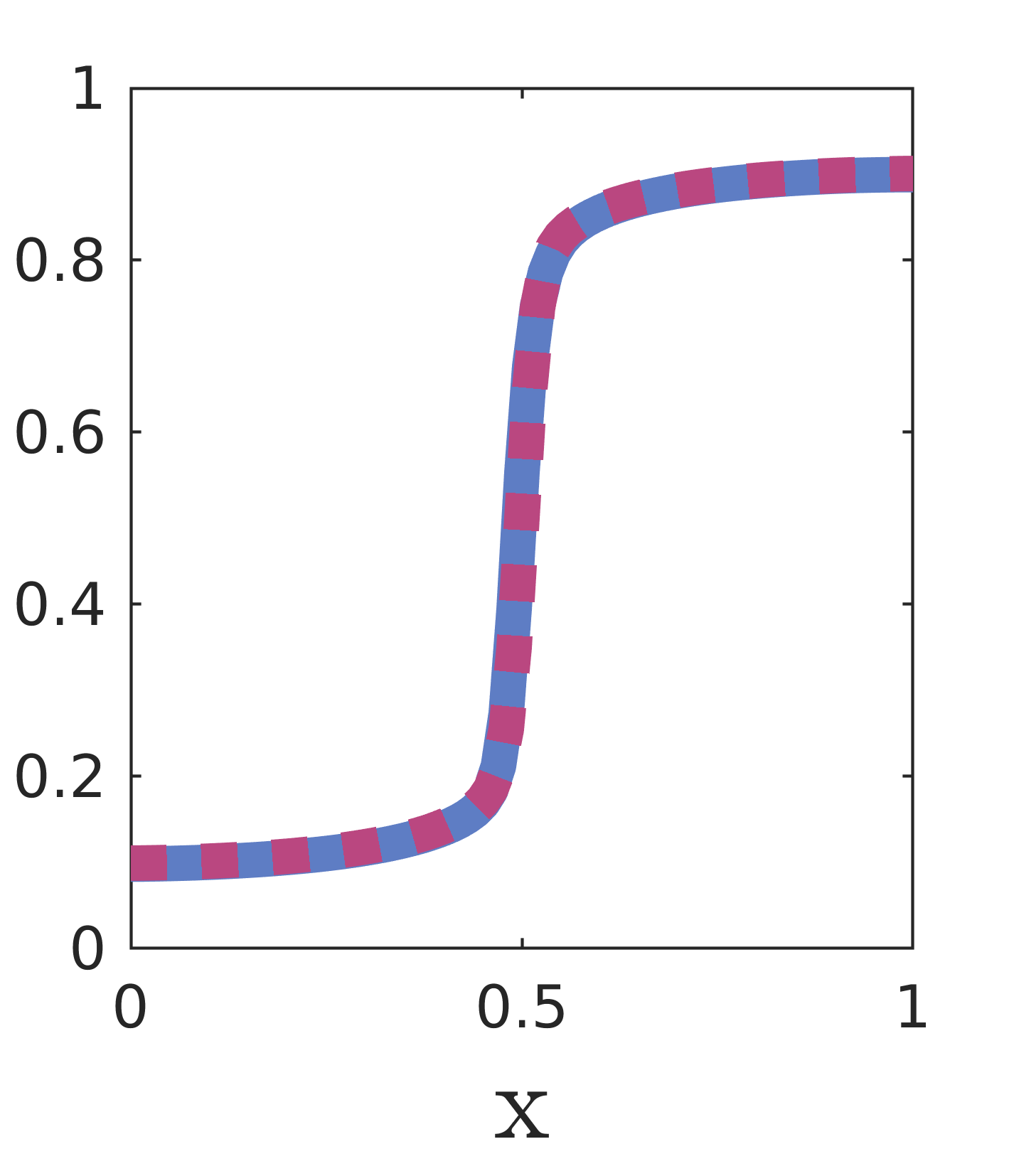}
}\hspace{0.5cm}\subfigure[$\delta=0.1$, $t=1000$]{
\includegraphics[width=2.5cm,height=3cm]{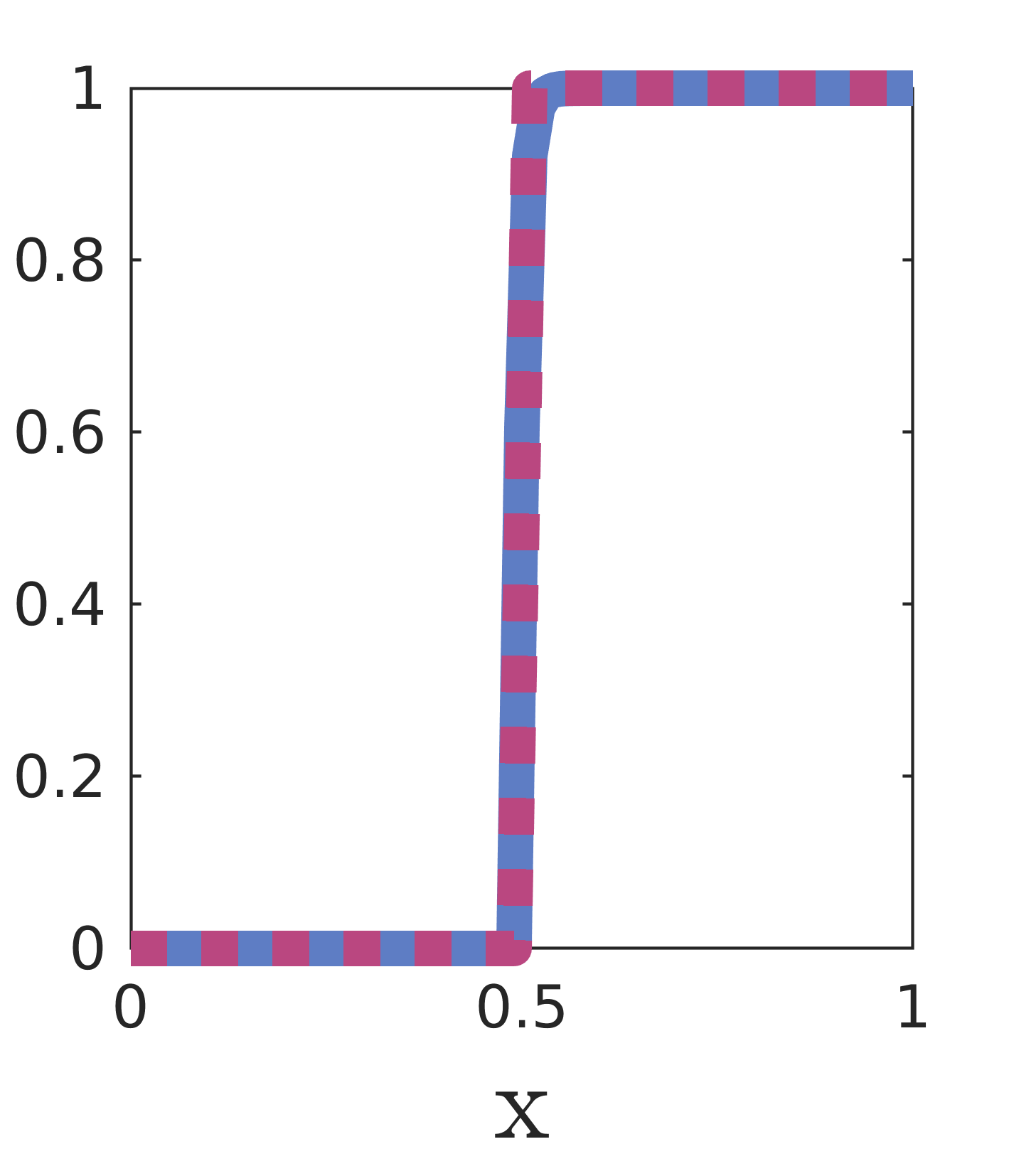}
}
\caption{Time evolution of $u_\epsilon(t,x)$ for different cell-diffusion coefficients in \eqref{eq:KSDC}, compared with the limiting profile $u$ of \eqref{hss2}, in one dimension with $m=5$, $\chi=100$, $D=1$, and $M=0.5$.}
\label{fig:label31311}
\end{figure}

\begin{figure}[ht]
\centering
\subfigure[$\chi=1$, $D=1$]{
\includegraphics[width=2.5cm,height=3cm]{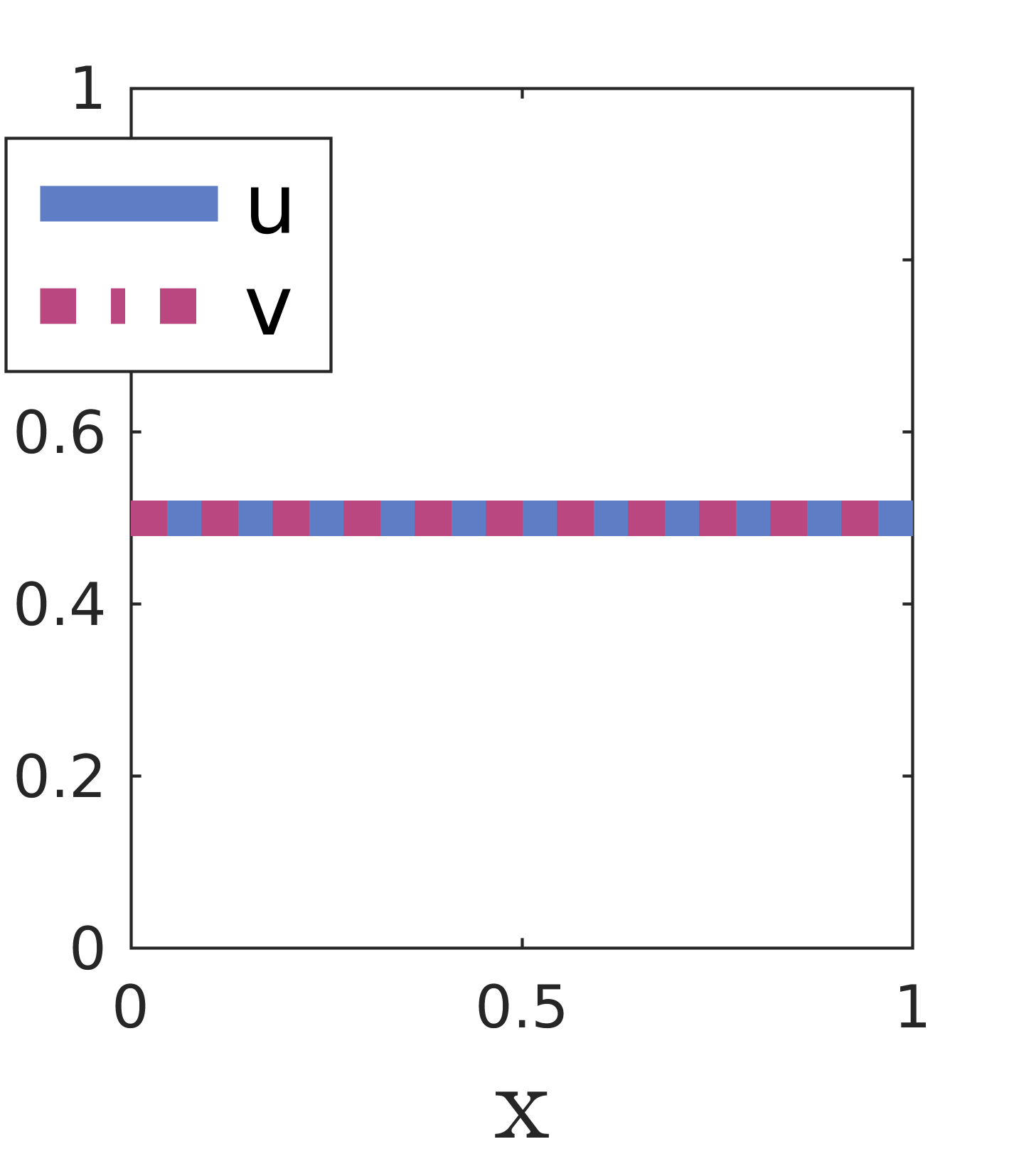}
}\hspace{0.5cm}\subfigure[$\chi=5$, $D=0.04$]{
\includegraphics[width=2.5cm,height=3cm]{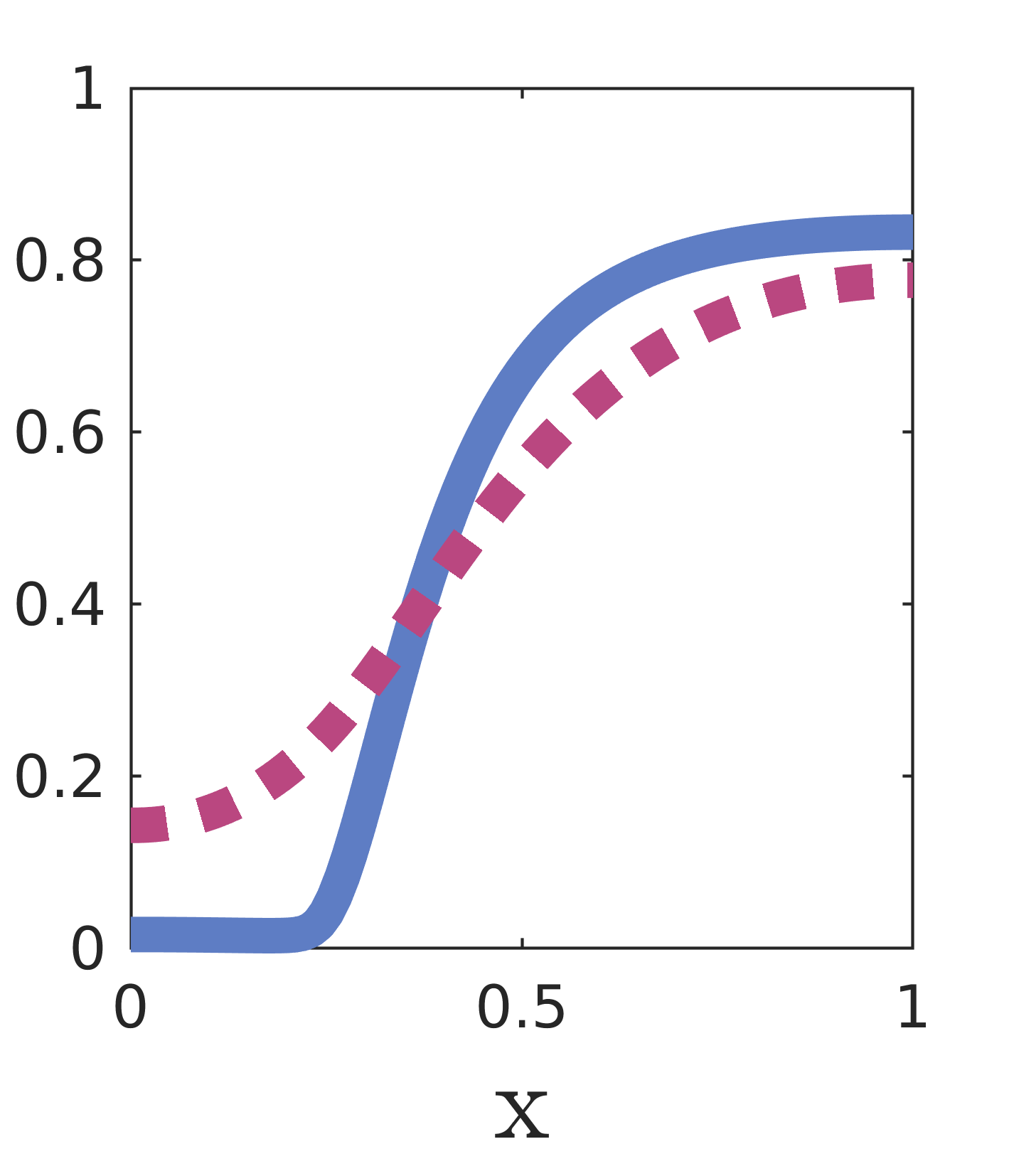}
}\hspace{0.5cm}\subfigure[$\chi=10$, $D=0.01$]{
\includegraphics[width=2.5cm,height=3cm]{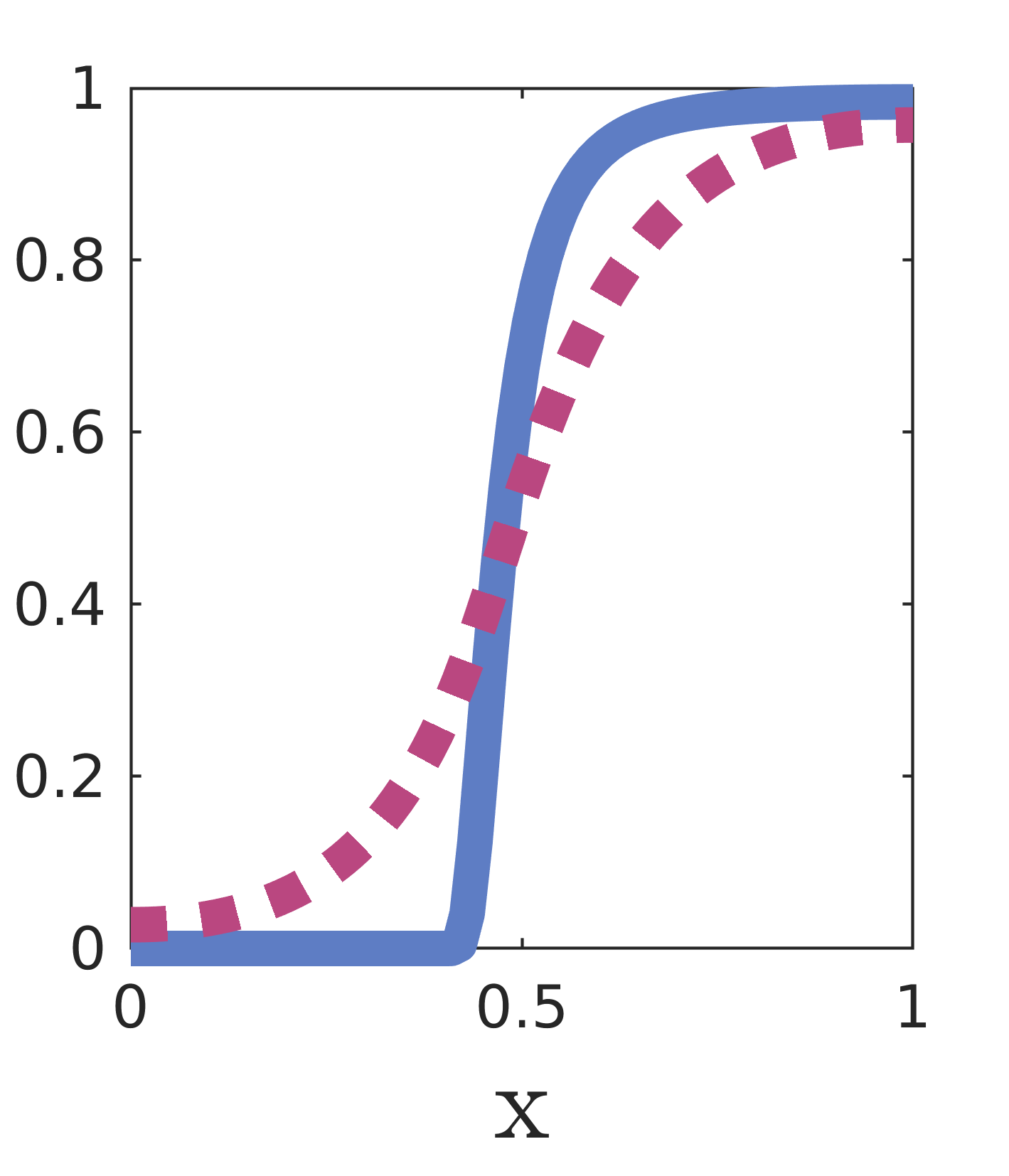}
}
\caption{Steady-state profiles under the scaling $\chi=\epsilon^{-1}$ and $D=\epsilon^2$ in one dimension, with $m=2$, $\delta=1$, and $M=0.5$.}
\label{fig:label}
\end{figure}

To compare the solutions $(u_\epsilon,v_\epsilon)$ of
system~\eqref{eq:KSDC} with the solutions of the corresponding limit
systems, we perform numerical simulations for different values of
$m$, $\delta$, $\chi$, and $D$.

In Figure~\ref{fig:label3131}, we choose $m$ and $\chi$ satisfying
the assumptions of Theorem~\ref{hks}. For $D=1$, the time evolution
of $u_\epsilon$ differs from the solution $u$ of the parabolic
system~\eqref{baru:equ}. When $D=0.01$, the two profiles are closer, which is consistent with the convergence established in
Theorem~\ref{hks}.

In Figure~\ref{fig:label31311}, we present the time evolution of the
density $u_\epsilon$ for different values of $\delta$. For large $\delta$, diffusion dominates and $u_\epsilon$ tends to a constant as time increases. For smaller $\delta$, the profile of $u_\epsilon$ is closer to the hyperbolic--elliptic limit, and the transition becomes sharper. This is consistent with Theorem~\ref{Theorem_7}.

In Figure~\ref{fig:label}, we show steady-state profiles of system
\eqref{eq:KSDC} under the scaling $\delta=1$,
$\chi=\epsilon^{-1}$, and $D=\epsilon^2$. As $\epsilon$ decreases, the density profile becomes sharper and approaches $0$ and $1$. This illustrates the sharp-interface behavior associated with the scaling in Theorem~\ref{Theorem10}.

\section{Conclusion and perspectives}

In this paper, we investigated three asymptotic regimes for system
\eqref{eq:KSDC}. A key finding is that the Keller--Segel system exhibits
three distinct asymptotic behaviors, depending on the choice of parameters.
This reveals interesting connections between different models. However,
several open questions regarding asymptotic limits remain unresolved.

In the first case, we established strong convergence to a scalar nonlinear
diffusion equation of porous-medium type as $D=\epsilon^2\to0$ when
$1\le m<2$ and $\chi$ satisfies the parabolicity condition. The regime
in which this condition fails merits further investigation. This regime may become more complex when the aggregation mechanism dominates,
potentially leading to oscillatory behavior
(see~\cite{Potapov2005,bpmz2024}), which poses challenges for proving
compactness of the density.

In the second case, as $\delta=\epsilon\to0$ with $D=1$, we proved that
the solutions converge, up to a subsequence, to a kinetic entropy solution
of the hyperbolic--elliptic Keller--Segel system \eqref{hss2}.

In the third case, we considered the limit of strong attraction
$\chi=\epsilon^{-1}\to\infty$ and small chemical diffusion
$D=\epsilon^2\to0$. For characteristic initial data of finite perimeter,
we showed that the solutions converge to the time-independent initial
characteristic function. A natural question is what happens if the
chemical diffusion is fixed, for instance $D=1$. One difficulty is that
the singular part of the energy~\eqref{energy1} is no longer uniformly
controlled when $u\ne v$.

\section*{Acknowledgements}

This work is part of the author's PhD thesis. The author would like to express her sincere gratitude to her PhD supervisor, Professor Beno\^\i t Perthame, for his support throughout this work.

\appendix
\section{A compactness lemma}\label{app:compactness}
The following Aubin--Lions-type compactness lemma is used in the proof of Theorem~\ref{Theorem10}; see \cite[Lemma B.1]{kim2024density}.
\begin{lemma}
Let $u_n$ be a sequence of functions bounded in $L^\infty(0,T; L^1(\Omega))$ such that $u_n$ is bounded in $L^\infty((0,T); \mathrm{BV}(\Omega))$ and $u_n \to u$ in $L^\infty((0,T); H^{-1}(\Omega))$. Then
\[
{\|u_n-u\|_{L^\infty(0,T;L^1(\Omega))}\to0.}
\]
\end{lemma}

%

%
%

\normalem
\bibliographystyle{plain}
\bibliography{MQBH_bib}

\end{document}